\RequirePackage{fix-cm}
\documentclass[smallextended]{svjour3}       
\smartqed  
\usepackage{graphicx}
\usepackage{hyperref}
\hypersetup{
    colorlinks=true,
    linkcolor=blue,
    filecolor=magenta,      
    urlcolor=cyan,
}
\usepackage{subcaption}
\usepackage{epsf}
\usepackage{amsmath,amssymb}
\usepackage{graphicx}
\usepackage{color}
\usepackage{cite}
\usepackage{multirow,tabularx}
\usepackage{bbm}
\newcommand{\brac}[1]{\left({#1}\right)}
\newcommand{\sbrac}[1]{\left[{#1}\right]}
\newcommand{\cbrac}[1]{\left\{{#1}\right\}}

\newcommand{\indic}[1]{\mathbbm{1}_{\cbrac{#1}}}

\newcommand{\expect}[2][]{\mathbb{E}_{#1}\left[{#2}\right]}

\newcommand{\prob}[2][]{\mb{P}_{#1}\brac{#2}}

\newcommand{\p}{\mathbb{P}}
\newcommand{\e}{\mathbb{E}}

\newcommand{\UN}[1]{{\mathcal{U}}^{(N)}_{j}}

\newcommand{\mb}[1]{\mathbb{#1}}

\newcommand{\floor}[1]{\left\lfloor{#1}\right\rfloor}

\begin{document}

\title{Voter and Majority Dynamics with Biased and Stubborn Agents
}

\titlerunning{Opinion Dynamics with Biased and Stubborn Agents}        

\author{Arpan~Mukhopadhyay         \and
        Ravi~R.~Mazumdar \and
        Rahul~Roy 
}

\authorrunning{Mukhopadhyay et al.} 

\institute{A.~Mukhopadhyay \at
              Department of Computer Science, UK\\
              University of Warwick \\
              \email{arpan.mukhopadhyay@warwick.ac.uk}           
           \and
           R.~R.~Mazumdar \at
           Department of Electrical and Computer Engg.\\
           University of Waterloo, Canada\\
           \email{mazum@uwaterloo.ca}
           \and
           R.~Roy \at
           Theoretical Statistics and Mathematics Unit\\
           Indian Statistical Institute, Delhi, India\\
           \email{rahul@isid.ac.in}
}

\date{Received: date / Accepted: date}

\maketitle

\begin{abstract}
We study binary opinion dynamics in a fully connected network of interacting agents. 
The agents are assumed to interact according to one of the following rules: (1) Voter rule:
An updating agent simply copies the opinion
of another randomly sampled agent;
(2) Majority rule: An updating agent samples multiple agents
and adopts the majority opinion 
in the selected group. 
We focus on the scenario where the agents are biased towards one of the opinions called the {\em preferred opinion}.
Using suitably constructed branching processes, we show that under both rules 
the mean time to reach consensus is $\Theta(\log N)$, where $N$ is the number of agents in the network.
Furthermore, under the majority rule model, we show
that consensus can be achieved on the preferred opinion with high probability even if 
it is initially the opinion of the minority. 
We also study the majority rule model when stubborn agents with fixed opinions are present.
We find that the stationary distribution of opinions in the network in the large system limit
using mean field techniques.
\end{abstract}

\section{Introduction}

The problem of social learning \cite{Ellison_JPE_1993,Chatterjee_AAP_2004}
is concerned with the rate at which social agents,
interacting under simple rules, can learn/discover the true utilities
of	 their choices, opinions or technologies. 
In this context, the two central questions we study are:
(1) Can social agents learn/adopt the better technology/opinion through simple rules
of interactions and if so, how fast? and (2) What are the effects of the presence of stubborn agents (having fixed
opinions) on the dynamics opinion diffusion?

We consider a setting 
where the choices available to each agent are binary and are represented by $\cbrac{0}$ and $\cbrac{1}$~\cite{Chatterjee_AAP_2004, Bandyopadhyay_EJP_2010}.
These are referred to as {\em opinions} of the agents.
The interactions among the agents are modelled using 
two simple rules: the {\em voter rule}~\cite{Ligett_voter_model, Sudbury_voter_model,Cox_voter_model} 
and the {\em majority rule}~\cite{Krapivsky_majority,Redner_majority,Chen_majority,ganesh_consensus}.
In the voter rule, an agent randomly samples one of its neighbours
at an instant when it decides to update its opinion. The updating agent 
then adopts the opinion of the sampled neighbour. 
This simple rule captures the tendency of an individual to
mimic other individuals in the society. 
In the majority rule, instead of sampling a single agent, an updating agent samples $2K$ ($K \geq 1$)
neighbours and adopts the opinion of the majority of the sampled neighbours (including itself).
This rule captures the tendency of the individuals to
conform with the majority opinion in their local neighbourhoods.

{\bf Related literature}: The voter model and its variants have been studied extensively (see \cite{emanuele_review} for a recent survey)
for different network topologies, e.g., finite integer lattices in different dimensions
~\cite{Cox_voter_model, Krapivsky_voter}, complete graphs with three states~\cite{three_state_voter_vojnovic}, heterogeneous graphs~\cite{Redner_voter},
random $d$-regular graphs~\cite{cooper_voter},
Erdos-Renyi random graphs, and random geometric graphs~\cite{Yildiz_ITA} etc.
It is known~\cite{peleg_voter,nakata_voter} that if the underlying graph is connected, then the classical voter rule leads
to a consensus where all agents adopt the same opinion. Furthermore, if $A$ is the set of 
all agents having an opinion $i\in \cbrac{0,1}$ initially, then the probability that consensus 
is achieved on opinion $i$ (referred to as the {\em exit probability} to opinion $i$)
is given by $d(A)/2m$, where $d(A)$ is the sum of the degrees
of the vertices in $A$ and $m$ is the total number of edges in the graph. 
It is also known that for most network topologies
the mean consensus time is $\Omega(N)$, where $N$ is the total number of agents. 
In~\cite{Mobilia_voter_stubborn,Yildiz_voter_stubborn}, the voter model was studied  
under the presence of stubborn individuals who do not update
their opinions.
In such a scenario, the network
cannot reach a consensus because of the presence
of stubborn agents having both opinions. 
Using coalescing random walk techniques
the average opinion in the network and the variance
of opinions were computed at steady state. 

The majority rule model was first studied in~\cite{Galam_majority}, where
it was assumed that, at every iteration, groups
of random sizes are formed by the agents. Within each group,
the majority opinion is adopted by all the agents. 
Similar models with fixed (odd) group size have been
considered in~\cite{Krapivsky_majority, Redner_majority}.
A more general majority rule based model is analysed in~\cite{ganesh_consensus} for complete graphs.
It has been shown that with high probability (probability tending to one as $N \to \infty$) consensus is achieved
on the opinion with the initial majority  and
the mean time to reach consensus time is $\Theta(\log N)$. 
The majority rule model is studied for random $d$-regular graphs on $N$ vertices
in~\cite{majority_regular_collin}. It is shown that when the initial 
imbalance of between the two opinions is above $c\sqrt{1/d+d/N}$, for some constant $c>0$,
then consensus is achieved in $O(\log N)$ time 
on the initial majority opinion with high probability.
A deterministic version of the majority rule model,
where an agent, instead of randomly sampling some of its neighbours,
adopts the majority opinion among all its neighbours, 
is considered in~\cite{Mossel_majority_new,Flocchini_majority_new,Moran_majority_new,AGUR1_majority}.
In such models, 
given the graph structure of the network,
the opinions of the agents at any time is a deterministic
function of the initial opinions of the agents.
The interest there is to find out the initial distribution
of opinions for which the network converges to some specific absorbing
state.

{\bf Contributions}: In all the prior works on the voter and the majority rule models, 
it is assumed that opinions or technologies are indistinguishable.
However, in a social learning model, one opinion/technology
may be inherently `better' than the other leading to more utility to individuals
choosing  the better option in a round of update. As a result, individuals
currently using the better technology will update less frequently than individuals with the worse technology.
To model this scenario, we assume that an agent having opinion $i \in \cbrac{0,1}$
performs an update with probability $q_i$. By choosing $q_1 < q_0$, we  make the agents `biased'
towards the opinion $\cbrac{1}$, which is referred to as the {\em preferred opinion}.
We study the opinion dynamics under both voter and majority rules when the agents are biased.
We focus on the case where the underlying graph is complete which closely models situations where agents are mobile 
and can therefore sample any other agent depending on their current neighbourhood.

For the voter model with biased agents, 
we show  that the probability of reaching consensus on the non-preferred
opinion decreases exponentially
with the network size.
Furthermore,  the mean  consensus time
is shown to be logarithmic in the network size. 
This is in sharp contrast to the voter model with unbiased agents where the probability of reaching consensus
on any opinion remains constant and the mean consensus time grows linearly with the network size.
Therefore, in the biased voter model consensus is achieved exponentially faster than that in the unbiased voter model.

For the majority rule model with biased agents, we show
that the network reaches consensus on the preferred opinion with high probability only if the initial fraction
of agents with the preferred opinion is above a certain threshold determined by the biases of the agents. 
Furthermore, as in the voter model,
the mean consensus time is shown to be logarithmic in the network size.
Our results generalise existing results on the majority rule model with unbiased agents 
where it is known that consensus is achieved (with high probability) on the opinion with the initial majority
and the mean consensus time is logarithmic in the network size. 
However, existing proofs for the unbiased majority rule model~\cite{Krapivsky_majority,ganesh_consensus}
cannot be extended to the biased case as they crucially rely on the fact that opinions are indistinguishable.
We use suitably constructed branching processes and 
monotonicity of certain rational polynomials to prove the results for the biased model.

We also study the majority rule model in the presence
of  agents having fixed opinions at all times.
These agents are referred to as 'stubborn' agents.
A similar study of the voter model in the presence of stubborn agents was done in~\cite{Yildiz_voter_stubborn}.
In presence of stubborn agents, the network cannot reach a consensus state. The key objective, therefore,
is to study the stationary distribution of opinions among the non-stubborn agents. 
In \cite{Yildiz_voter_stubborn}, coalescing random
random walk techniques were used to study this stationary distribution of opinions.
However, such techniques do not apply to majority rule dynamics. We analyse the network dynamics in the large scale limit 
using mean field techniques. In particular, we show that depending on the proportions of stubborn agents the mean field can either have
single or multiple equilibrium points. If multiple equilibrium points are present, the network shows {\em metastability}
in which it switches between stable configurations spending long time in each configuration.

An earlier version of this work~\cite{arpan_opinion_ITC}, 
contained some of the results of this paper and an analysis of the majority rule model for $K=1$.
However, only sketches of the proofs were provided.
In the current paper, we provide
rigorous proofs of all results and a more general analysis of the majority rule model (for $K \geq 1$).

{\bf Organisation}: The rest of the paper is organised as follows.
In Section~\ref{sec:voter_models}, we introduce
the model with biased agents. In Sections~\ref{sec:voter_asymmetric}
and~\ref{sec:majority_biased}, we state the main results for the voter model and the majority rule model
with biased agents, respectively. Section~\ref{sec:majority_stubborn} analyses the majority rule
model with stubborn agents. In Sections~\ref{pf:voter_consensus}-\ref{pf:majority_consensus}, we provide
the detailed proofs the main results on voter and majority rule models with biased agents.
Finally, the paper is concluded in Section~\ref{sec:opinion_conclusion}.

\section{Model with biased agents}
\label{sec:voter_models}

We consider a network of $N$ social agents. 
Opinion of each agent is assumed to be a binary variable
taking values in the set $\{0, 1\}$.  
Initially, every agent adopts one of the two opinions.
Each agent considers updating its opinion at points of 
an independent unit rate Poisson point process associated with itself. 
At a point of the Poisson process associated with itself, an agent
either updates its opinion or retains its past opinion.
We assume that an agent with opinion $i \in \cbrac{0,1}$
updates its opinion at a point of the unit rate Poisson
process associated with itself with probability $q_i \in (0,1)$ and 
retains its opinion with probability $p_i=1-q_i$.
To make the agents `biased' towards opinion $\cbrac{1}$ we assume
that $q_0 > q_1$ which implies that an agent with opinion $\cbrac{1}$
updates its opinion less frequently than an agent with opinion $\cbrac{0}$.

In case the agent decides to update its opinion, it 
does so either using the voter rule or under the majority rule.
In the voter rule, an updating agent samples 
an agent uniformly at random from $N$ agents 
(with replacement) from the network\footnote{In the large $N$ limit 
sampling with or without replacement leads to the same results.}
and adopts the opinion of the sampled agent.
In the majority rule, an updating agent samples $2K$ agents ($K \geq 1$)
uniformly at random (with replacement) and adopts the opinion of the majority
of the $2K+1$ agents including itself.
The results derived in this paper
can be extended to the case where the updating agent samples an agent 
from a random group of size $O(N)$. 
However, for simplicity we only focus on the case where sampling occurs from the whole population.

%


\section{Main results for the voter model with biased agents}
\label{sec:voter_asymmetric}


We first consider the voter model with biased agents.
In this case,
clearly, the network reaches consensus in a finite time with probability 1. 
Our interest is to find out the probability with which consensus is achieved 
on the preferred opinion $\{1\}$. This is referred to as the 
{\em exit probability} of the network.
We also intend to characterise
the average time to reach the consensus. 

The case $q_1=q_0=1$ is referred to
as the voter model with unbiased agents, which has been analysed in~\cite{Ligett_voter_model,Sudbury_voter_model}.  
It is known that for unbiased agents
the  probability  with which consensus is
reached on a particular opinion is simply equal to the initial fraction $\alpha$
of agents having that opinion  and the expected time to reach consensus
for large $N$ is approximately given by $N h(\alpha)$, where
$h(\alpha)=-[\alpha \ln(\alpha)+(1-\alpha)\ln(1-\alpha)]$.
We now proceed to characterise these quantities for the 
voter model with biased agents.

Let $X^{(N)}(t)$
denote the number of agents
with opinion $\{1\}$ at time $t \geq 0$. 
Clearly, $X^{(N)}(\cdot)$ is a Markov process on state space $\{0,1,\ldots,N\}$,
with absorbing states $0$ and $N$.
The transition rates from state $k$ are given by

\begin{align}
q(k\to k+1)&=q_0 k \frac{N-k}{N},\\
q(k\to k-1)&=q_1 k \frac{N-k}{N},
\end{align}
where $q(i \to j)$ denotes the rate of transition from state $i$ to state $j$.
The embedded discrete-time Markov chain $\tilde{X}^{(N)}$
for $X^{(N)}$ is a one-dimensional random walk on $\cbrac{0,1,\ldots,N}$
with jump probability of $p=q_0/(q_0+q_1)$ to the right and
$q=1-p$ to the left. We define $r=q/p <1$ and $\bar r=1/r$.
Let $T_k$ denote the first hitting time of state $k$, i.e.,

\begin{equation}
T_k=\inf\cbrac{t\geq 0: X^{(N)}(t) \geq k},
\end{equation}
We are interested in the asymptotic behaviour of the quantities
$E_N(\alpha):=\p_{\floor{\alpha N}}\brac{T_N < T_0}$
and $t_N(\alpha)=\e_{{\floor{\alpha N}}}\sbrac{T_0 \wedge T_N}$, where $\p_x\brac{\cdot}$
and $\e_x\sbrac{\cdot}$, respectively denote the probability measure and expectation
conditioned on the event $X^{(N)}(0)=x$.
To characterise the above quantities we require the following lemma which
follows from the gambler ruin identity for one-dimensional asymmetric random walks \cite{Spitzer_book}.

\begin{lemma}
\label{lem:stop}
For $0 \leq a < x < b \leq N$, we have

\begin{equation}
\p_x\brac{T_a < T_b}=\frac{r^x-r^b}{r^a-r^b}.
\end{equation}
\end{lemma}

%
%

%
%

From the above lemma 
it follows that
\begin{align*}
E_N(\alpha)&= \p_{\floor{\alpha N}}\brac{T_N < T_0},\\
				   &=\frac{1-r^{\floor{\alpha N}}}{1-r^N},\\
				   &\geq 1-r^{\floor{\alpha N}}
				   =1-\exp(-cN),
\end{align*}
for some constant $c >0$ (since $r < 1$). Hence,  
the probability of having a consensus on the non-preferred opinion 
approaches $0$ exponentially fast in $N$. This is unlike the voter model
with unbiased agents where the probability of having consensus on either opinion
remains constant with respect to $N$.

The following theorem characterises 
the mean time ${t}_N(\alpha)$ to reach the consensus
state starting from $\alpha$ fraction of agents having opinion $\{1\}$.

\begin{theorem}
\label{thm:voter_consensus}
For all $\alpha \in (0,1)$ we have $t_N(\alpha)=\Theta(\log N)$.
\end{theorem}

Hence, the above theorem shows that the mean consensus time in the biased
voter model is logarithmic in the network size.
This is in contrast to the voter model with unbiased agents
where the mean consensus time is linear in the network size. Thus, with biased
agents, the network reaches consensus exponentially faster.

We now consider the 
measure-valued process $x^{(N)}=X^{(N)}/N$,
which describes the evolution of the fraction of agents with opinion $\{1\}$.
We show that the following convergence takes place.

\begin{theorem}
\label{thm:mean_field}
If $x^{(N)}(0) \Rightarrow \alpha$, then $x^{(N)}(\cdot) \Rightarrow x(\cdot)$,
where $\Rightarrow$ denotes weak convergence and
$x(\cdot)$ is a deterministic process with initial condition $x(0)=\alpha$
and is governed by the following differential equation

\begin{equation}
\label{eq:mean_field}
\dot{x}(t) = (q_0-q_1)x(t)(1-x(t)),
\end{equation}
\end{theorem}

According to the above result, for large $N$, the process $x^{(N)}(\cdot)$
is well approximated by the deterministic process $x(\cdot)$ which is 
generally referred to as the {\em mean field limit} of the system. 
Using the mean field limit, we can
approximate the mean consensus time $t_N(\alpha)$ by the time the process $x(\cdot)$
takes to reach the state $1-1/N$ starting from $\alpha$.

\begin{theorem}
\label{thm:voter_estimate}
\begin{enumerate}
\item For the process $x(\cdot)$ defined by \eqref{eq:mean_field}, we have 
$x(t) {\to} 1$ as $t \to \infty$ for any $x(0) \in (0,1)$.

\item Let $t(\epsilon,\alpha)$ denote the time required by the process $x(\cdot)$
to reach $\epsilon \in (0,1)$ starting from $x(0)=\alpha \in (0,1)$. Then

\begin{equation}
t(\epsilon,\alpha)=\frac{1}{q_0-q_1}\brac{\log \frac{\epsilon}{1-\epsilon}-\ln \frac{\alpha}{1-\alpha}}.
\label{eq:time}
\end{equation}
In particular, for $\epsilon=1-1/N$ we have

\begin{align}
t(1-1/N,\alpha)&=\frac{1}{q_0-q_1} \log(N-1)-\frac{1}{q_0-q_1} \log \brac{\frac{\alpha}{1-\alpha}}.
\label{eq:cons_time_voter_asym}
\end{align}
\end{enumerate}
\end{theorem}

\begin{proof}
Since $ q_0 > q_1$ and $x(t) \in (0,1)$ for all $t \geq 0$, 
we have from~\eqref{eq:mean_field} that $\dot{x}(t) \geq 0$ for all $t \geq 0$.
Hence, $x(t) \to 1$ as $t \to \infty$.

The second assertion follows directly by solving \eqref{eq:mean_field}
with initial condition $x(0)=\alpha$. \qed
\end{proof}
%
%
%

\begin{remark}
It is worth noting here that
the process $x(\cdot)$ does not reach $1$ in a finite time
even though the process $x^{(N)}(\cdot)$ does reach $1$
in a finite time with probability $1$.
However, it is `reasonable' to assume that $t_N(\alpha)$ 
is 'closely' approximated by $t(1-1/N,\alpha)$. 
Such approximation of the absorption time
of an absorbing Markov chain using its corresponding mean field limit is 
common in the literature \cite{three_state_voter_vojnovic,Krapivsky_majority}. 
However, except from a few special cases, e.g. \cite{gast_absorption}, there is no general theory justifying
such approximations. 
\end{remark}

\begin{figure}[h!]
 \centering
 \includegraphics[height=0.35\textwidth,keepaspectratio]{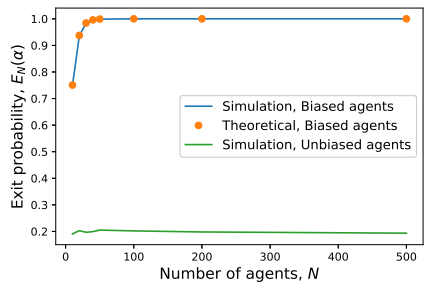}
 \caption{Exit probability $E_N(\alpha)$
 as a function of the number of agents $N$.
 Parameters: $q_0=1$, $q_1=0.5$, $\alpha=0.2$.}
 \label{fig:exit_asymmetric}
 \end{figure}

{\bf Simulation Results}: In Figure~\ref{fig:exit_asymmetric}, we plot
the exit probability for both unbiased ($q_0=q_1=1$) and biased ($1=q_0 > q_1=0.5$) cases 
as functions of the number of agents $N$ 
for $\alpha=0.2$. As expected from our theory, we observe that
in the biased case the exit probability exponentially increases to $1$ with the increase $N$. 
This is in contrast to
the unbiased case, where the exit probability remains constant
at $\alpha$ for all $N$.

In Figure~\ref{fig:cons_time_voter_asymmetric}, 
we plot the mean consensus time ${t}_N(\alpha)$
for both unbiased and biased cases as a function of $N$ for $\alpha=0.4$. 
We observe a good match between the estimate obtained in Theorem~\ref{thm:voter_estimate}
and the simulation results.
The observation also verifies the statement of Theorem~\ref{thm:voter_consensus}.
In Figure~\ref{fig:cons_time_voter_asymmetric_alpha}, we plot
the mean consensus time as a function of $\alpha$
for both biased an unbiased cases. The network size is kept fixed at $N=100$.
We observe that for the unbiased case,
the consensus time increases in the range $\alpha \in (0,0.5)$ and decreases in the range $\alpha \in (0.5,1)$.
In contrast, for the biased case, 
the consensus time steadily decreases with the increase
in $\alpha$. This is expected since, in the unbiased case,
consensus is achieved faster on a particular opinion if the
initial number agents having that opinion is more than
the initial number of agents having the other opinion.
On the other hand, in the biased case, consensus is 
achieved with high probability on the preferred opinion and therefore
increasing the initial fraction of agents having the preferred opinion
always decreases the mean consensus time.



\begin{figure*}[t!]
    \centering
    \begin{subfigure}[t]{0.5\textwidth}
        \centering
        \includegraphics[width=\textwidth]{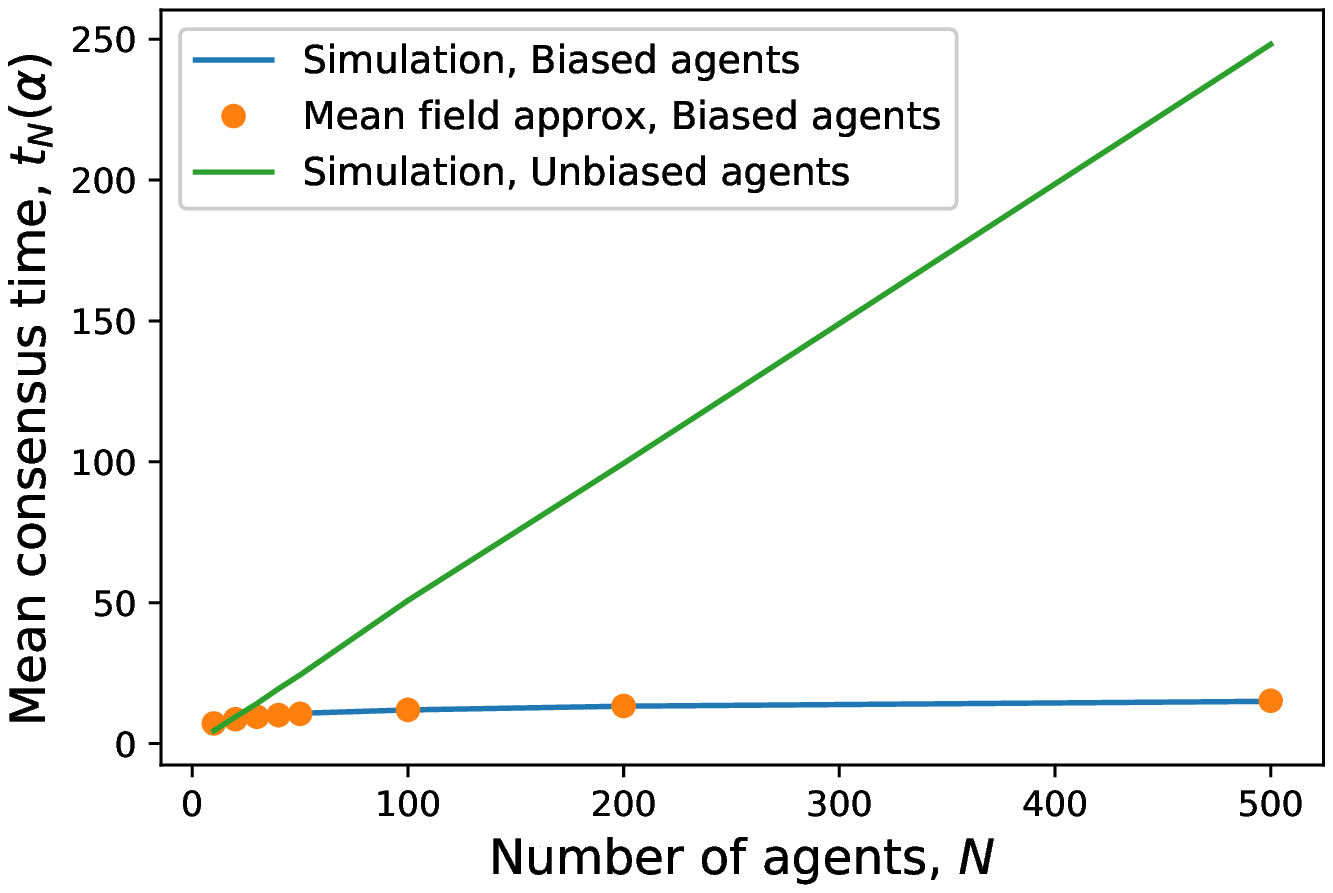}
        \caption{Mean consensus time ${t}_N(\alpha)$ as a function of the number of agents $N$}
        \label{fig:cons_time_voter_asymmetric}
    \end{subfigure}%
    ~ 
    \begin{subfigure}[t]{0.5\textwidth}
        \centering
        \includegraphics[width=\textwidth]{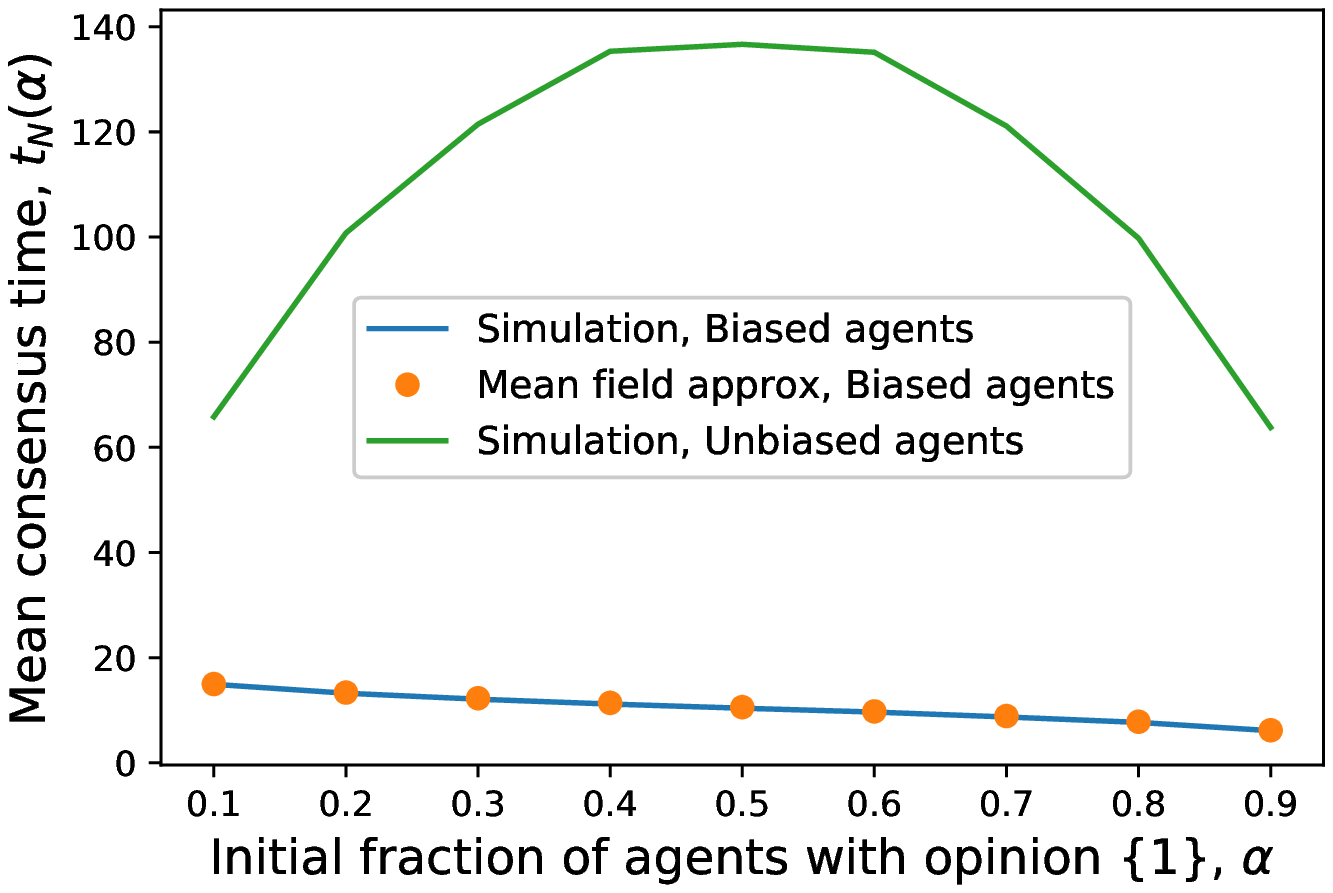}
        \caption{Mean consensus time ${t}_N(\alpha)$ as a function of the initial fraction $\alpha$ of agents with opinion $\{1\}$.}
        \label{fig:cons_time_voter_asymmetric_alpha}
    \end{subfigure}
    \caption{Mean consensus time under the voter model with biased agents}
\end{figure*}

\section{Main results for the majority rule model with biased agents}
\label{sec:majority_biased}

In this section, 
we consider the majority rule model with biased agents.
As in the voter model, it is easy to see that in this case a consensus is achieved
in a finite time with probability 1.
We proceed to find the exit probability to opinion ${1}$ and the mean consensus time.

Let $X^{(N)}(t)$
denote the number of agents
with opinion $\{1\}$ at time $t \geq 0$. 
Clearly, $X^{(N)}(\cdot)$ is a Markov process on state space $\{0,1,\ldots,N\}$.
The jump rates of $X^{(N)}$ from state $n$ to
state $n+1$ and $n-1$ are given by

\begin{align}
q(n \to n+1)&=(N-n)q_0\sum_{i=K+1}^{2K}\binom{2K}{i} \brac{\frac{n}{N}}^i \brac{\frac{N-n}{N}}^{2K-i},\\
                    &=(N-n)q_0\prob{\text{Bin}\brac{2K,\frac{n}{N}} \geq K+1},\label{eq:uprate}\\
q(n \to n-1)&=n q_1\sum_{i=K+1}^{2K}\binom{2K}{i} \brac{\frac{N-n}{N}}^i \brac{\frac{n}{N}}^{2K-i},\\
                   &=nq_1\prob{\text{Bin}\brac{2K,1-\frac{n}{N}} \geq K+1},\label{eq:downrate}
\end{align}
respectively. Let $\tilde X^{(N)}$ denote the embedded Markov chain corresponding to $X^{(N)}$.
Then the jump probabilities for the embedded chain $\tilde X^{(N)}$
are given by 
\begin{align}
p_{n,n+1}&=1-p_{n,n-1}\nonumber\\
                &=\frac{q(n \to n+1)}{q(n \to n-1)+q(n \to n+1)}\nonumber\\
                &=\frac{g_K(n/N)}{g_K(n/N)+r},
\label{eq:pdef}
\end{align}
where $1\leq n \leq N-1$, 
$r=q_1/q_0 < 1$, and $g_K:(0,1) \to (0,\infty)$ is defined as

\begin{equation}
g_K(x)=\frac{\frac{1}{x}\prob{\text{Bin}(2K,x) \geq K+1}}{\frac{1}{1-x} \prob{\text{Bin}\brac{2K,1-x} \geq K+1}}.
\label{eq:defg}
\end{equation}
With probability $1$, $X^{(N)}$ gets absorbed in one of the states
$0$ or $N$ in finite time. We are interested in the probability of absorption in the
state $N$ and the average time till absorption. We first state the following lemma
which is key to proving many of the results in this section.

\begin{lemma}
\label{lem:monotone}
The function $g_K:(0,1) \to (0,\infty)$ as defined by \eqref{eq:defg}
is strictly increasing and is therefore also one-to-one.
\end{lemma}

In the following theorem, we characterise the exit probability to state $N$.

\begin{theorem}
\label{thm:phase_transit}
\begin{enumerate}
\item Let $E_N(n)$ denote the probability that the process $X^{(N)}$ gets absorbed in
state $N$ starting from state $n$. Then, we have

\begin{equation}
E_N(n)=\frac{\sum_{t=0}^{{n}-1}  \prod_{j=1}^t \frac{r}{g_K(j/N)}}{\sum_{t=0}^{N-1} \prod_{j=1}^t \frac{r}{g_K(j/N)}},
\label{eq:exit_majority_asymmetric}
\end{equation}

\item Define $E_N(\alpha):=E_N(\floor{\alpha N})$ and $\beta:=g_K^{-1}(r)$.
Then  $E_N(\alpha) \to 1$ (resp. $E_N(\alpha) \to 0$) as $N \to \infty$
if $\alpha > \beta$ (resp. $\alpha < \beta$) and this convergence is exponential in $N$.
\end{enumerate}
\end{theorem}

%
%
%
%
%
%
%
Hence, a phase transition of the exit probability occurs at $\beta=g_K^{-1}\brac{r}$
for all values of $K \geq 1$. This implies, that even though the agents are biased
towards the preferred opinion, consensus may not be obtained on the preferred opinion if the initial fraction of agents having the preferred opinion
is below the threshold $\beta$. This is in contrast to the voter model, where consensus is obtained on the preferred opinion irrespective of the initial state.  
The threshold $\beta$ can be computed by solving $g_K(\beta)=r$ using either Newton-Raphson method or other fixed point methods.

\begin{remark}
We note that for the unbiased majority rule model we have $r=1$
and $\beta=g_K^{-1}(r)=g_K^{-1}(1)=1/2$. 
Thus, the known results \cite{ganesh_consensus,Krapivsky_majority} 
for the majority
rule model with unbiased agents are recovered.
\end{remark}

We now characterise 
the mean time ${t}_N(\alpha)$ to reach the consensus
state starting from $\alpha$ fraction of agents having opinion $\{1\}$.
As before, we define $T_n$ to be the random time
of first hitting the state $n$, i.e., $T_n=\inf\cbrac{t\geq 0: X^{(N)}(t)\geq n}$.

\begin{theorem}
\label{thm:majority_consensus}
For $\alpha \in (0,\beta) \cup (\beta,1)$ we have $t_N(\alpha)=\Theta(\log N)$.
\end{theorem}

The theorem above shows that the mean consensus time is logarithmic in the network size.
To prove the theorem we use branching processes and the monotonicity shown in Lemma~\ref{lem:monotone}.
Our proof does not require the indistinguishability of the opinions and is therefore
more general than existing proofs for the unbiased voter model\cite{Krapivsky_majority,ganesh_consensus}.

It is easy to derive the mean field limit corresponding to
the empirical measure process $x^{(N)}=X^{(N)}/N$.
Using the transition rates of the process $X^{(N)}(\cdot)$
it can be verified that 
if $x^{(N)}(0) \Rightarrow \alpha$ as $N \to \infty$, then
$x^{(N)}(\cdot) \Rightarrow x(\cdot)$ as $N \to \infty$, where process
$x(\cdot)$ satisfies the initial condition $x(0)=\alpha$ and is governed by
the following ODE:
%

\begin{align}
\dot{x}(t)&=q_0x(t) (1-x(t))h_K(x(t))(g_K(x(t))-r), 
\label{eq:maj_K}
\end{align}
where $h_K$ is defined as
$h_K(x)=\sum_{i=K+1}^{2K} \binom{2k}{i} (1-x)^{i-1}x^{2K-i}$.
By definition $h_K(x) > 0$ for $x \in (0,1)$.
Hence, from Lemma~\ref{lem:monotone}, it follows that
the process $x(\cdot)$ has three equilibrium points
at $0$, $1$, and $\beta$, respectively.
Furthermore, using the monotonicity of $g_K$ established in Lemma~\ref{lem:monotone}
and the non-negativity of $h_K$ we have that $\dot x(t) > 0$ for $x(t) > \beta$
and $\dot x(t) < 0$ for $x(t) < \beta$.
This shows that the only stable equilibrium points of 
the mean field limit $x(\cdot)$ are $0$ and $1$.
At $\beta$, $x(\cdot)$ has an unstable equilibrium point.

{\bf Simulation Results}:  In Figure~\ref{fig:exit_majority_asymmetric}, 
we plot the exit probability $E_N(\alpha)$
as a function of the total number $N$ of agents in the network.
The parameters are chosen to be $q_0=1$, $q_1=0.6$, $K=1$.
For this parameter setting, we can explicitly compute the threshold $\beta$ to be $\beta=g_K^{-1}(r)=q_1/(q_0+q_1)=0.375$.
We observe that for $\alpha > \beta$ the exit probability exponentially increases to $1$
with the increase in $N$
and for $\alpha < \beta$ the exit probability decreases exponentially to zero with the increase in $N$.
This is in accordance with the assertion made in Theorem~\ref{thm:phase_transit}.
Similarly, in Figure~\ref{fig:exit_majority_asymmetric_alpha},
we plot the exit probability as a function of the initial
fraction $\alpha$ of agents having opinion $\{1\}$ for the same parameter setting and different
values of $N$.
The plot shows a clear phase transition at $\beta=0.375$.
The sharpness of the transition increases as $N$ increases.

\begin{figure*}[t!]
    \centering
    \begin{subfigure}[t]{0.5\textwidth}
        \centering
        \includegraphics[width=\textwidth]{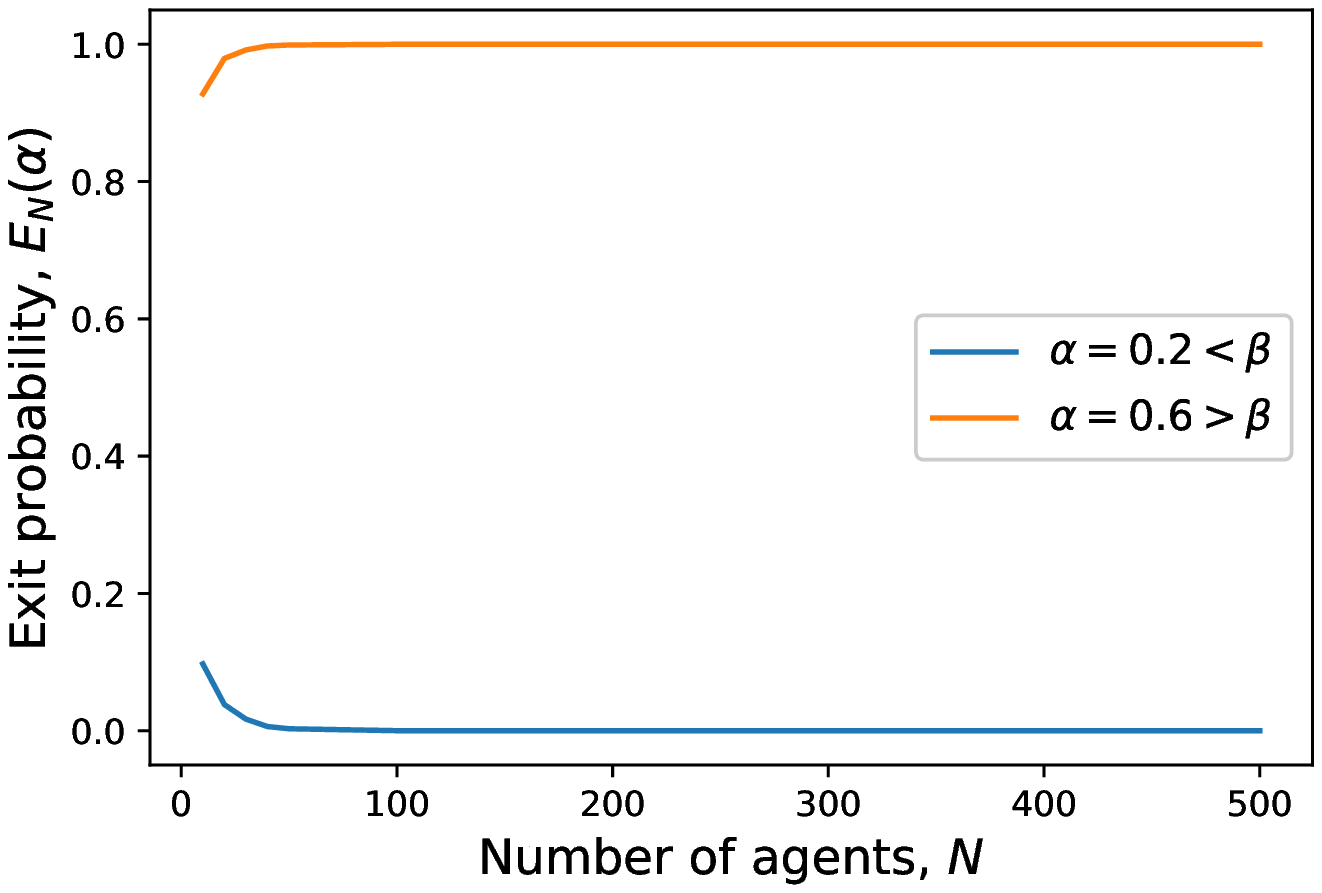}
        \caption{Exit probability $E_N(\alpha)$
 as a function of the number of agents $N$.
 Parameters: $q_0=1$, $q_1=0.6$}
        \label{fig:exit_majority_asymmetric}
    \end{subfigure}%
    ~ 
    \begin{subfigure}[t]{0.5\textwidth}
        \centering
        \includegraphics[width=\textwidth]{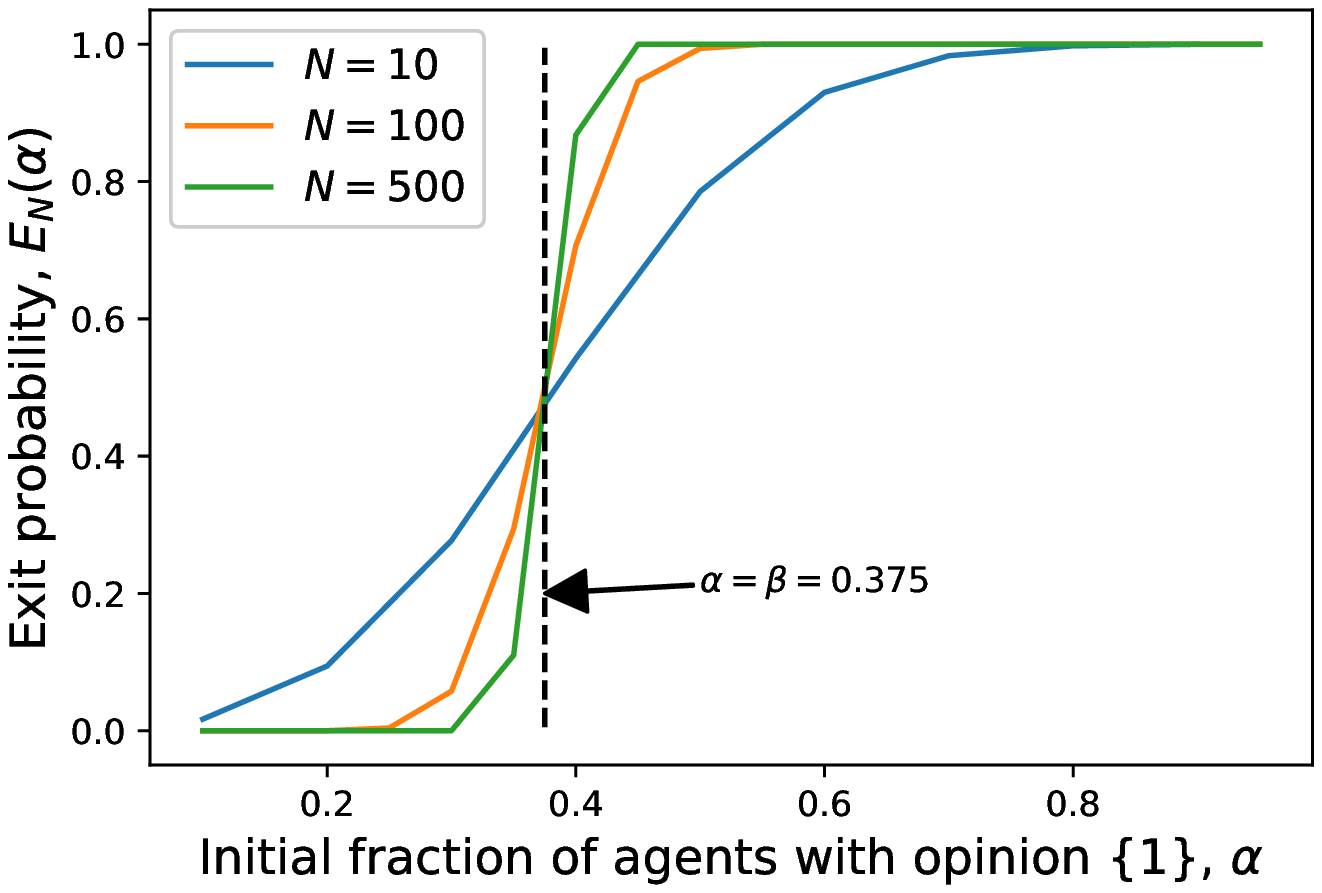}
        \caption{Exit probability $E_N(\alpha)$ as a function of the initial fraction $\alpha$
 of agents with opinion $\{1\}$. Prameters: $q_0=1, q_1=0.6$.}
        \label{fig:exit_majority_asymmetric_alpha}
    \end{subfigure}
    \caption{Mean consensus time under the voter model with biased agents}
    \label{fig:exit_prob_maj_asym}
\end{figure*}

In Figure~\ref{fig:consensus_N_majority}, we plot the mean consensus time under the majority rule
as a function of $N$ for different values of $\alpha$. As predicted by Theorem~\ref{thm:majority_consensus},
we find that that the mean consensus time is logarithmic in the network size.
In Figure~\ref{fig:consensus_majority_asymmetric_K},
we study the mean consensus time as a function of $K$
for $q_0=1,q_1=0.6,\alpha=0.5, N=50$.
We observe that with the increase in $K$, the mean time to reach consensus  decreases.
This is expected since the slope of the mean field $x(t)$
increases with $K$. This leads to faster convergence to the stable equilibrium points.


\begin{figure*}[t!]
    \centering
    \begin{subfigure}[t]{0.5\textwidth}
        \centering
        \includegraphics[width=\textwidth]{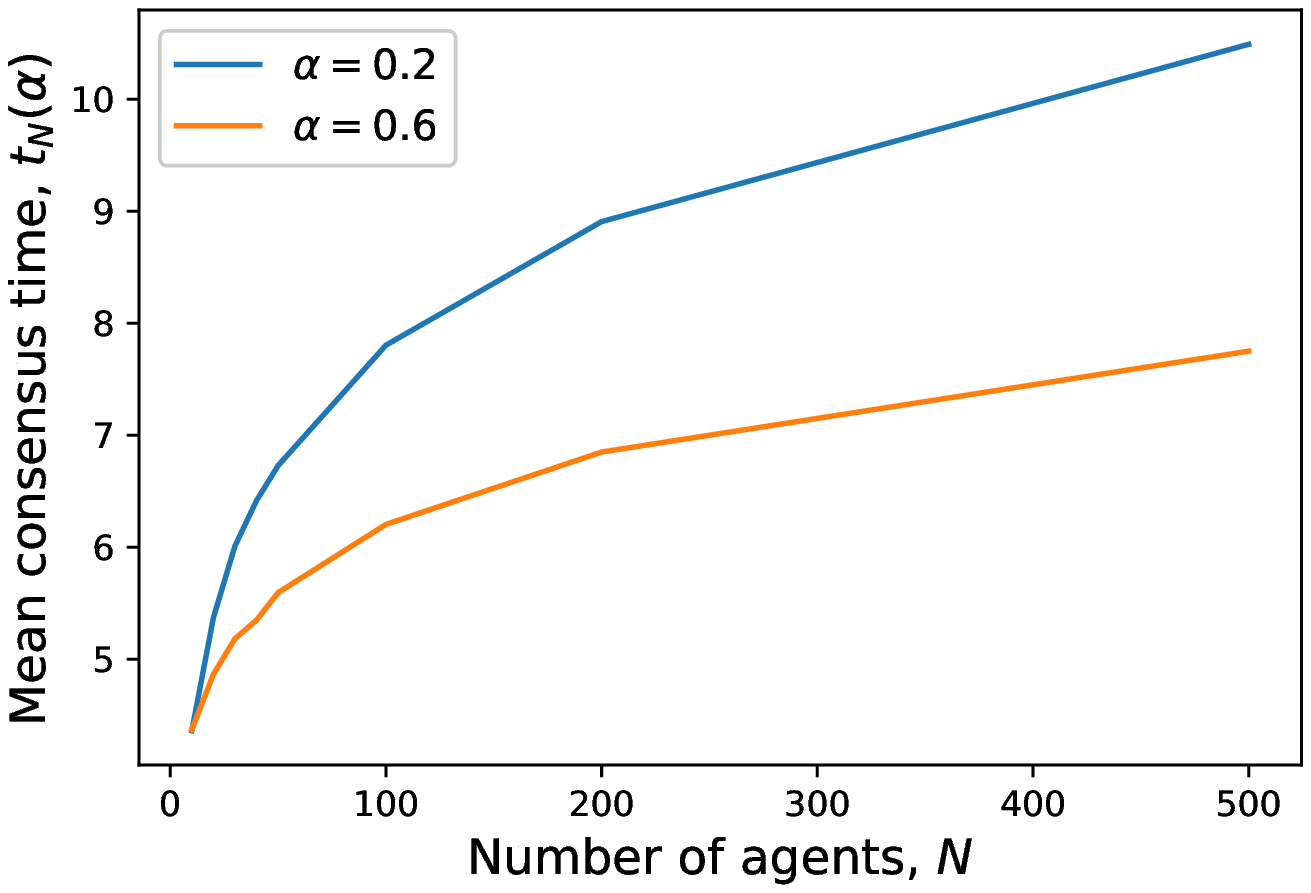}
        \caption{Mean consensus time as a function of $N$. Parameters: $q_0=1,q_1=0.6,\alpha=0.5$}
        \label{fig:consensus_N_majority}
    \end{subfigure}%
    ~ 
    \begin{subfigure}[t]{0.5\textwidth}
        \centering
        \includegraphics[width=\textwidth]{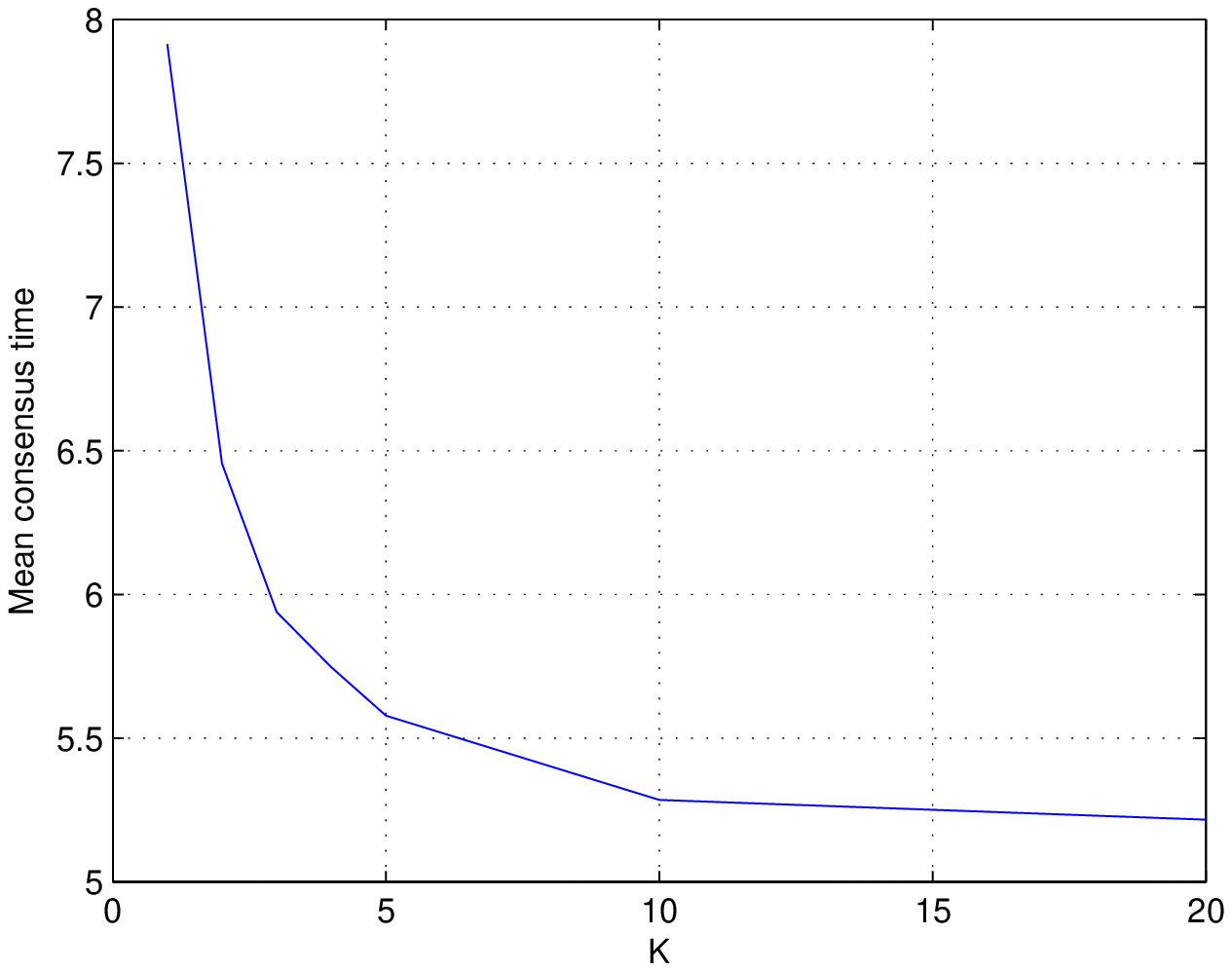}
        \caption{Mean consensus time as a function of $K$.
 Parameters: $q_0=1,q_1=0.6,\alpha=0.5, N=50$}
        \label{fig:consensus_majority_asymmetric_K}
    \end{subfigure}
    \caption{Mean consensus time under majority rule with biased agents}
    \label{fig:exit_prob_maj_asym}
\end{figure*}

\section{Majority model with stubborn agents}
\label{sec:majority_stubborn}

In this section, we consider the majority rule model in the presence of `stubborn agents'.
These are agents that never update their opinions.
The other agents, referred to as the {\em non-stubborn} agents,
are assumed to update their opinions at all points of the Poisson
processes associated with themselves. 
We focus on the case where the updates occur according
to the majority rule model. The voter model with stubborn agents
was studied before in \cite{Yildiz_voter_stubborn} using coalescing random walks. 
However, this technique does not apply to
the majority rule model. We use mean field techniques to study the opinion dynamics under the majority rule model.
 
We denote by $\gamma_i$, $i \in \{0,1\}$, the fraction of agents in network who are stubborn
and have opinion $i$ at all times.
Thus, $(1-\gamma_0-\gamma_1)$
is fraction of non-stubborn agents in the network.
The presence of stubborn agents prevents
the network
from reaching a consensus state.
This is because at all times there are at least $N \gamma_0$
stubborn agents having opinion $\{0\}$ and $N \gamma_1$ stubborn agents
having opinion $\{1\}$. 
Furthermore, since each non-stubborn agent
may interact with some stubborn agents at every update instant,
it is always possible for the non-stubborn agent to change its opinion. 
Below we characterise the equilibrium fraction
of non-stubborn agents having opinion $\{1\}$ in the network
for large $N$ using mean field techniques. 
For analytical tractability, we consider the case $K=1$, i.e.,
when an agent sample two agents at each update instant. 
However, similar results hold even for larger values of $K$.

Let $x^{(N)}(t)$ denote the fraction of non-stubborn agents
having opinion $\{1\}$ at time $t \geq 0$. Clearly, $x^{(N)}(\cdot)$
is a Markov process with possible jumps at the 
points of a rate $N(1-\gamma_0-\gamma_1)$
Poisson process. 
The process $x^{(N)}(\cdot)$
jumps from the state $x$ to the state $x+1/N(1-\gamma_0-\gamma_1)$
when one of the non-stubborn agents having opinion $\{0\}$
becomes active (which happens with rate $N(1-\gamma_0-\gamma_1)(1-x)$)
and samples two agents with opinion $\{1\}$. The probability of sampling an agent
having opinion $\{1\}$ from the entire network is $(1-\gamma_0-\gamma_1)x+\gamma_1$.
Hence, the total rate at which the process transits from state $x$
to the state $x+1/N(1-\gamma_0-\gamma_1)$ is given by

\begin{multline}
q\brac{x \to x+\frac{1}{N(1-\gamma_0-\gamma_1)}}=N(1-\gamma_0-\gamma_1)(1-x)\\
\times [(1-\gamma_0-\gamma_1)x+\gamma_1]^2.
\end{multline}
Similarly, the rate of the other possible transition
is given by

\begin{multline}
q\brac{x \to x-\frac{1}{N(1-\gamma_0-\gamma_1)}}=N(1-\gamma_0-\gamma_1)x\\
\times  [(1-\gamma_0-\gamma_1)(1-x)+\gamma_0]^2.
\end{multline}
As in Theorem~\ref{thm:mean_field}, it can be shown from the above transition rates
that the process $x^{(N)}(\cdot)$
converges weakly to the mean field limit $x(\cdot)$
which satisfies the following differential equation

\begin{multline}
\dot{x}(t)=(1-x(t))[(1-\gamma_0-\gamma_1)x(t)+\gamma_1]^2\\
-x(t)
[(1-\gamma_0-\gamma_1)(1-x(t))+\gamma_0]^2.
\label{eq:mean_field_majority_stubborn}
\end{multline}
%
We now study the equilibrium distribution $\pi_N$
of the process $x^{(N)}(\cdot)$ for large $N$ via the equilibrium
points of the mean field $x(\cdot)$.

From~\eqref{eq:mean_field_majority_stubborn} we see that $\dot{x}(t)$
is a cubic polynomial in $x(t)$. Hence, the process $x(\cdot)$
can have at most three equilibrium points in $[0,1]$.
We first characterise the stability of these equilibrium points.

\begin{proposition}
\label{prop:meta_stubborn}
The process $x(\cdot)$ defined by~\eqref{eq:mean_field_majority_stubborn}
has at least one equilibrium point in $(0,1)$. Furthermore,
the number of stable equilibrium points of $x(\cdot)$ in $(0,1)$ is either two or one.
If there exists only one equilibrium point of $x(\cdot)$ in $(0,1)$, then the equilibrium point
must be globally stable (attractive).
\end{proposition}

\begin{proof}
Define $f(x)=(1-x)[(1-\gamma_0-\gamma_1)x+\gamma_1]^2-x[(1-\gamma_0-\gamma_1)(1-x)+\gamma_0]^2$.
%
%
Clearly, $f(0)=\gamma_1^2 > 0$ and $f(1)=-\gamma_0^2 < 0$.
Hence, there exists at least one root of $f(x)=0$ in $(0,1)$.
This proves the existence of an equilibrium point of $x(\cdot)$ in $(0,1)$.

Since $f(x)$ is a cubic polynomial and $f(0)f(1) < 0$,
either all three roots of $f(x)=0$ lie in $(0,1)$ or exactly one root of $f(x)=0$
lies in $(0,1)$.
Let the three (possibly complex and non-distinct) 
roots  of $f(x)=0$ be denoted by $r_1, r_2, r_3$, respectively.
By expanding $f(x)$ we see that
the coefficient of the cubic term is $-2(1-\gamma_0-\gamma_1)^2$.
Hence, $f(x)$ can be written as

\begin{equation}
f(x)=-2(1-\gamma_0-\gamma_1)^2(x-r_1)(x-r_2)(x-r_3).
\label{eq:mean_field_temp}
\end{equation}

We first consider the case when $0 < r_1, r_2, r_3 <1$ and not all of them are equal.
Let us suppose, without loss of generality, that
the roots are arranged in the increasing order, i.e., $0 < r_1 \leq r_2 < r_3 < 1$ or $0 < r_1 < r_2 \leq r_3 < 1$.
From~\eqref{eq:mean_field_temp} and~\eqref{eq:mean_field_majority_stubborn},
it is clear that,
if $x(t)> r_2$ and $x(t) > r_3$, then $\dot{x}(t) < 0$. Similarly, if $x(t)> r_2$ and $x(t) < r_3$,
then $\dot{x}(t) > 0$ .
Hence, if $x(0) > r_2$ then $x(t) \to r_3$ as $t \to \infty$.
Using similar arguments we have that for $x(0) < r_2$, $x(t) \to r_1$ as $t \to \infty$.
Hence,  $r_1,r_3$
are the stable equilibrium points of $x(\cdot)$. 
This proves that there exist at most two stable
equilibrium points of the mean field $x(\cdot)$.

Now suppose that there exists only one equilibrium point of $x(\cdot)$
in $(0,1)$.
This is possible either i) if there exists exactly one  real root of $f(x)=0$
in $(0,1)$, or ii) if all the roots of $f(x)=0$ are equal and lie in $(0,1)$.
Let $r_1$ be a root of $f(x)=0$ in $(0,1)$.
Now by expanding $f(x)$ from~\eqref{eq:mean_field_temp}, we see that
the product of the roots must be $\gamma_1^2/2(1-\gamma_0-\gamma_1)^2 >0$.
This implies that the other roots, $r_2$ and $r_3$,
must satisfy one of the following conditions: 1) $r_2, r_3 >1$,
2) $r_2, r_3 < 0$, 3) $r_2, r_3$ are complex conjugates, 4) $r_2=r_3=r_1$.

%
%
%
%
%
In all the above cases, we have
that $(x-r_2)(x-r_3) \geq 0$ for all $x \in [0,1]$ with equality if and only if $x=r_1=r_2=r_3$.
Hence, from~\eqref{eq:mean_field_temp} and~\eqref{eq:mean_field_majority_stubborn},
it is easy to see that $\dot{x}(t) > 0$ when $0 \leq x(t) < r_1$
and $\dot{x}(t) < 0$ when $1 \geq x(t) > r_1$.
This implies that $x(t) \to r_1$ for all $x(0) \in [0,1]$.
In other words, $r_1$ is globally stable. 
\end{proof}

In the next proposition, we provide the conditions on $\gamma_0$ and $\gamma_1$
for which there exist multiple stable equilibrium points of the mean field $x(\cdot)$.

\begin{proposition}
\label{prop:meta_condition}
There exist two distinct stable equilibrium points of
the mean field $x(\cdot)$ in $(0,1)$ if and only if

\begin{enumerate}
\item $D(\gamma_0,\gamma_1)=(\gamma_0-\gamma_1)^2+3(1-2\gamma_0-2\gamma_1) > 0$
\item $0 < z_1, z_2 < 1$, where

\begin{align}
z_1&= \frac{(3-\gamma_0-5\gamma_1)+ \sqrt{D(\gamma_0,\gamma_1)}}{6(1-\gamma_0-\gamma_1)},\\
z_2&= \frac{(3-\gamma_0-5\gamma_1)- \sqrt{D(\gamma_0,\gamma_1)}}{6(1-\gamma_0-\gamma_1)}.
\end{align}

\item $f(z_1)f(z_2) \leq 0$, where $f(x)=(1-x)[(1-\gamma_0-\gamma_1)x+\gamma_1]^2-x[(1-\gamma_0-\gamma_1)(1-x)+\gamma_0]^2$.
\end{enumerate}
If any one of the above conditions is not satisfied 
then $x(\cdot)$ has a unique, globally stable
equilibrium point in $(0,1)$.
\end{proposition}

\begin{proof}
From Proposition~\ref{prop:meta_stubborn}, we have seen that $x(\cdot)$ has two stable equilibrium points
in $(0,1)$ if and only if $f(x)=0$ has three real roots in $(0,1)$ among which
at least two are distinct.
This happens if and only if $f'(x)=0$
has two distinct real roots $z_1, z_2$ in the interval $(0,1)$ and
$f(z_1) f(z_2) \leq 0$. Since $f'(x)$ is a quadratic polynomial
in $x$, the above conditions are satisfied if and only if
%
%
%
\begin{enumerate}
\item The discriminant of $f'(x)=0$ is positive. This corresponds
to the first condition of the proposition.

\item The two roots $z_1,z_2$ of $f'(x) =0$ must lie in $(0,1)$.
This corresponds to the second condition of the proposition.

\item $f(z_1) f(z_2) \leq 0$. This is
the third condition of the proposition. 
\end{enumerate}

Clearly, if any one of the above conditions is not satisfied,
then $x(\cdot)$ has a unique equilibrium point in $(0,1)$.
According to Proposition~\ref{prop:meta_stubborn} this equilibrium
point must be globally stable.
\end{proof}

Hence,
depending on the values of $\gamma_0$ and $\gamma_1$
there may exist of multiple
stable equilibrium points of the mean field $x(\cdot)$. However,
for every finite $N$, the process $x^{(N)}(\cdot)$
has a unique stationary distribution $\pi_N$
(since it is irreducible on a
finite state space). 
In the next result, we establish that
any limit point of the sequence of stationary
probability distributions $(\pi_N)_N$
is a convex combination of the Dirac measures concentrated on the
equilibrium points of the mean field $x(\cdot)$ in $[0,1]$.

\begin{theorem}
Any limit point of the sequence of probability measures $(\pi_N)_N$
is a convex combination of the Dirac measures 
concentrated on the equilibrium points of $x(\cdot)$ in $[0,1]$.
In particular, if there exists a unique equilibrium point 
$r$ of $x(\cdot)$ in $[0,1]$
then $\pi_N \Rightarrow \delta_r$, where $\delta_{r}$ denotes the
Dirac measure concentrated at the point $r$. 
\end{theorem}

\begin{proof}
We first note that since the sequence of probability measures $(\pi_N)_N$
is defined on the compact space $[0,1]$, it must be tight. Hence,
Prokhorov's theorem implies that $(\pi_N)_N$ is relatively compact.
Let $\pi$ be any limit point of the sequence $(\pi_N)_N$. Then by
the mean field convergence result we know that $\pi$
must be an invariant distribution of the maps $\alpha \mapsto x(t,\alpha)$
for all $t \geq 0$, i.e., $\int \varphi(x(t,\alpha))d\pi(\alpha)=\int \varphi(\alpha)d\pi(\alpha)$,
%
%
for all $t \geq 0$, and all continuous (and hence bounded)
functions $\varphi: [0,1] \mapsto \mb{R}$. In the above,
$x(t,\alpha)$ denotes the process $x(\cdot)$ started at $x(0)=\alpha$.  
Hence we have

\begin{align}
\int \varphi(\alpha)d\pi(\alpha) &= \lim_{t \to \infty} \int \varphi(x(t,\alpha))d\pi(\alpha)\\
                                 &= \int \varphi\brac{\lim_{t \to \infty} x(t,\alpha)}d\pi(\alpha).
                                 \label{eq:inter_conv}
\end{align}
The second equality follows from the first by the Dominated
convergence theorem and the continuity of $\varphi$.
Now, let $r_1, r_2$, and $r_3$ denote the three equilibrium
points of the mean field $x(\cdot)$. Hence, by Proposition~\ref{prop:meta_stubborn}
we have that for each $\alpha \in [0,1]$, $\varphi(\lim_{t \to \infty}x(t,\alpha))
=\varphi(r_1) I_{N_{r_1}}(\alpha)+\varphi(r_2) I_{N_{r_2}}(\alpha)+\varphi(r_3) I_{N_{r_3}}(\alpha)$, where
for $i=1,2,3$, $N_{r_i} \in [0,1]$ denotes the set for which if $x(0) \in N_{r_i}$
then $x(t) \to r_i$ as $t \to \infty$, and $I$ denotes the indicator function.
Hence, by~\eqref{eq:inter_conv} we have that for all continuous functions
$\varphi: [0,1] \mapsto \mb{R}$

\begin{equation}
\int \varphi(\alpha)d\pi(\alpha) = \varphi(r_1) \pi(N_{r_1})+\varphi(r_2) \pi(N_{r_2})+\varphi(r_3) \pi(N_{r_3}).
\end{equation}
This proves that $\pi$ must be of the form
$\pi=c_1 \delta_{r_1}+c_2 \delta_{r_2}+c_3 \delta_{r_3}$,
where $c_1, c_2, c_3 \in [0,1]$ are such that $c_1+c_2+c_3=1$.
This completes the proof.
\end{proof}

Thus, according to the above theorem, if 
there exists a unique equilibrium point of the process $x(\cdot)$ in [0,1],
then the sequence of stationary distributions $(\pi_N)_N$
concentrates on that equilibrium point as $N \to \infty$. 
In other words, for large $N$, the fraction
of non-stubborn agents having opinion $\{1\}$ (at equilibrium) will
approximately be equal to the unique equilibrium point of the mean field.

{\bf Simulation Results}:
In Figure~\ref{fig:equilibrium_majority_stubborn},
we plot the equilibrium point of $x(\cdot)$ (when it is unique) as a function
of the fraction $\gamma_1$ of agents
having opinion $\{1\}$ who are stubborn keeping the fraction $\gamma_0$ of stubborn agents
having opinion $\{0\}$ fixed.
We choose the parameter values so that there exists a unique equilibrium point of $x(\cdot)$
in $[0,1]$ (such parameter settings can be obtained using the conditions of
Proposition~\ref{prop:meta_condition}).
We see that as $\gamma_1$ is increased in the range $(0,1-\gamma_0)$,
the equilibrium point shifts closer to unity. This is expected since increasing
the fraction of stubborn agents with opinion $\{1\}$ increases
the probability with which a non-stubborn agent samples
an agent with opinion $\{1\}$ at an update instant.

\begin{figure}
 \centering
 \includegraphics[width=60mm]{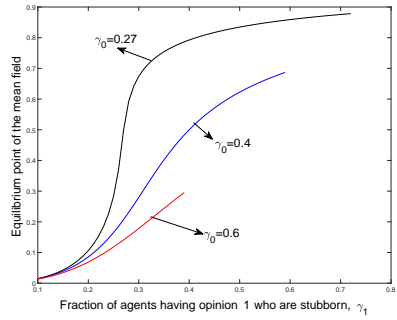}
 \caption{Majority rule with stubborn agents: Equilibrium point of $x(\cdot)$  as a function of
 $\gamma_1$ for different values of $\gamma_0$.}
 \label{fig:equilibrium_majority_stubborn}
 \end{figure}
 
 If there exist multiple equilibrium points of the process $x(\cdot)$
then the convergence $x^{(N)}(\cdot) \Rightarrow x(\cdot)$ implies
 that at steady state the process $x^{(N)}(\cdot)$
spends intervals
near the region corresponding to one of the stable equilibrium points of $x(\cdot)$. 
Then due to some
rare events, it reaches, via the unstable equilibrium point,
to a region corresponding to the other stable equilibrium point of $x(\cdot)$.
This fluctuation repeats giving the process $x^{(N)}(\cdot)$ a unique stationary distribution.
This behavior is formally known as {\em metastability}.

To demonstrate metastability, we simulate a network with $N=100$
agents and $\gamma_0=\gamma_1=0.2$. For the above parameters, the mean field $x(\cdot)$ 
has two stable equilibrium points
at $0.127322$ and $0.872678$.
In Figure~\ref{fig:sample_path_meta}, we show the sample path of the process
$x^{(N)}(\cdot)$. We see that at steady state the process switches back and forth
between regions corresponding to the stable equilibrium points of $x(\cdot)$.
This provides numerical evidence of the metastable behavior of the finite system.

\begin{figure}
 \centering
 \includegraphics[width=60mm]{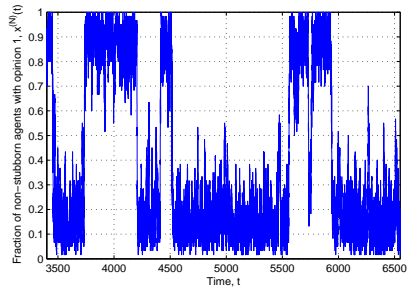}
 \caption{Majority rule with stubborn agents: Sample path of the process $x^{(N)}(\cdot)$ with
 $N=100$, $\gamma_0=\gamma_1=0.2$.}
 \label{fig:sample_path_meta}
 \end{figure}
 
\section{Proof of Theorem~\ref{thm:voter_consensus}}
\label{pf:voter_consensus}

Let $T=T_0\wedge T_N$ denote the random time to reach 
consensus.
Then we have

\begin{align}
T&=\sum_{k=1}^{N-1}\sum_{j=1}^{Z_k} M_{k,j},
\end{align}
where $Z_k$ denotes the number of visits to state $k$ before absorption
and $M_{k,j}$ denotes the time spent in the $j^{\textrm{th}}$ visit to state $k$.
Clearly, the random variables $Z_k$ and $(M_{k,j})_{j\geq 1}$ are independent with each
$M_{k,j}$ being an exponential random variable with rate $(q_0+q_1)k(N-k)/N$.
Hence, using Wald's identity we have

\begin{align}
t_{N}(\alpha)&=\e_{\floor{\alpha N}}\sbrac{T}\\
					&=\sum_{k=1}^{N-1} \e_{\floor{\alpha N}}\sbrac{Z_k} \e_{\floor{\alpha N}}\sbrac{M_{k,j}}\\
				    &=\frac{1}{q_0+q_1}\sum_{k=1}^{N-1}\brac{\frac{1}{k}+\frac{1}{N-k}} \e_{\floor{\alpha N}}\sbrac{Z_k}. \label{eq:consensus}
\end{align}
We now proceed to find lower and upper bounds of $t_N(\alpha)$.

Let $A=\cbrac{\omega: T_N(\omega) < T_0(\omega)}$ denote the event that the Markov chain
gets absorbed in state $N$. 
We have

\begin{align}
\label{eq:num_visits_break}
\expect[\floor{\alpha N}]{Z_k}=\expect[\floor{\alpha N}]{Z_k \vert A} \prob[\floor{\alpha N}]{A}
                     &+\expect[\floor{\alpha N}]{Z_k \vert A^c} (1-\prob[\floor{\alpha N}]{A}).
\end{align}

\noindent {\em Lower bound of $t_N(\alpha)$}: We first obtain a lower bound for $t_N(\alpha)$.
Clearly, we have the following

\begin{align*}
Z_k \vert A & \geq 1 \quad \forall k \geq \floor{\alpha N}\\
Z_k \vert A^c & \geq  0 \quad \forall k > \floor{\alpha N}\\
Z_k \vert A & \geq 0 \quad \forall k < \floor{\alpha N}\\
Z_k \vert A^c & \geq 1 \quad \forall k \leq \floor{\alpha N}.
\end{align*}
Using the above in~\eqref{eq:num_visits_break} we have
\begin{align*}
\expect[\floor{\alpha N}]{Z_k}& \geq  \prob[\floor{\alpha N}]{A}\indic{k \geq \floor{\alpha N}}+ (1-\prob[\floor{\alpha N}]{A})\indic{k \leq \floor{\alpha N}},
\end{align*}
where $\mathbbm{1}_{\Omega}$ denotes the indicator function for the set $\Omega$.
Using the above in \eqref{eq:consensus}, we have

\begin{align}
t_{N}(\alpha)&\geq \frac{\prob[\floor{\alpha N}]{A}}{q_0+q_1}\sum_{k=\floor{\alpha N}}^{N-1}\brac{\frac{1}{k}+\frac{1}{N-k}}+ \frac{1-\prob[\floor{\alpha N}]{A}}{q_0+q_1}\sum_{k=	1}^{\floor{\alpha N}}\brac{\frac{1}{k}+\frac{1}{N-k}}\\
					&\geq \frac{1}{q_0+q_1} \sum_{k=1}^{N(\alpha \wedge (1-\alpha))} \frac{1}{k}\\
					& > \frac{1}{q_0+q_1} \log(N(\alpha \wedge (1-\alpha))).
\end{align}

\noindent {\em Upper bound for $t_{N}(\alpha)$}: We first obtain
an upper bound on $\e_x\sbrac{Z_k;A}$ for $k \geq x$ with any $0 < x < N$.
Given $A$,
let $\zeta_k$ denote the number of times the embedded chain $\tilde X^{(N)}$ jumps
from $k$ to $k-1$ before absorption.
It is easy to observe that conditioned on $A$, the embedded chain 
$\tilde X^{(N)}$ is a Markov chain with jump
probabilities given by

\begin{equation}
p^{A}_{k,k+1}=1-p^{A}_{k,k-1}=p\frac{\prob[k+1]{A}}{\prob[k]{A}}\overset{(a)}{=}p\frac{1-r^{k+1}}{1-r^{k}},
\label{eq:cond_jump}
\end{equation}
where equality (a) follows from Lemma~\ref{lem:stop}.
Furthermore, we have

\begin{equation}
\label{eq:num_visits}
Z_k\vert A=
			1+\zeta_{k}+\zeta_{k+1},  \text{ for } x \leq k \leq N-1.
\end{equation}
The above relationship follows by observing that the facts that 
(i) the states $k \geq x$ are visited at least once and
(ii) the number of visits to state $k$ 
is the sum of the numbers of jumps of $\tilde X^{(N)}$ to the left and to the right
from state $k$.

Given $A$, we must have $\zeta_N=0$.
Let $\xi_{l,k}$ denote the random number of left-jumps from state $k$ between 
$l^{\textrm{th}}$ and $(l+1)^{\textrm{th}}$ left-jumps from state $k+1$. Then 
$(\xi_{l,k})_{l\geq 0}$ are i.i.d  with geometric distribution having mean $p^{A}_{k,k-1}/p^{A}_{k,k+1}$.
Moreover, we have the following recursion

\begin{equation}
\zeta_{k}=	
			\sum_{l=0}^{\zeta_{k+1}} \xi_{l,k},  \text{ for } x \leq k \leq N-1,
\end{equation}
(the above sum starts from $l=0$ because for $k \geq x$ 
left jumps from $j$ can occur even before the chain visits $j+1$ for the first time)
Thus, we see that $(\zeta_k)_{x\leq k \leq N}$ forms a branching process
with immigration of one individual in each generation. 
Applying Wald's identity to solve the above recursion we have
for $x \leq k \leq N-1$

\begin{align}
\e_x\sbrac{\zeta_{k}}&=	
			\sum_{n=k}^{N-1} \prod_{i=k}^{n} \frac{p_{i,i-1}^A}{p_{i,i+1}^A} \nonumber\\
			&\overset{(a)}{=}\sum_{n=k}^{N-1} \prod_{i=k}^{n} \frac{r(1-r^{i-1})}{1-r^{i+1}} \nonumber\\
			&=\frac{r(1-r^{k-1})(1-r^{N-k})}{(1-r)(1-r^N)}, \label{eq:left_jump}
\end{align}
where equality (a) follows from \eqref{eq:cond_jump}.
Taking expectation in \eqref{eq:num_visits} and substituting \eqref{eq:left_jump}
we obtain that for $x \leq k \leq N-1$

\begin{align}
\e_x\sbrac{Z_k\vert A}&=1+\e_x\sbrac{\zeta_{k}}+\e_x\sbrac{\zeta_{k+1}} \nonumber\\
								   &=\frac{1+r}{1-r}\frac{(1-r^{N-k})(1-r^k)}{1-r^N} \nonumber\\
								   &\leq \frac{1+r}{1-r}.
\label{eq:upper_half}
\end{align}
But we also have $$\e_x\sbrac{Z_k;A}\leq \e_x\sbrac{Z_k\vert A}\leq \frac{1+r}{1-r}.$$
which provides the required bound on $\e_x\sbrac{Z_k; A}$.
We note that the above bound is independent of $x$. 
In particular, it is true when $x=\floor{\alpha N}$ and $k \geq \floor{\alpha N}$.

For $1 \leq k < x$ we have

\begin{align*}
\e_x\sbrac{Z_k; A}&= \e_x\sbrac{Z_k; T_k < T_N < T_0},\\
					&=\e_x\sbrac{Z_k\vert T_k < T_N < T_0} \p_x\brac{T_k < T_N < T_0},\\
					&\overset{(a)}{=} \e_k\sbrac{Z_k\vert T_N < T_0} \p_x\brac{T_k < T_N < T_0},\\
					& \overset{(b)}{\leq} \frac{1+r}{1-r},
\end{align*}
where the equality $(a)$ follows from the Markov property
and inequality $(b)$ follows from \eqref{eq:upper_half}.
Hence, combining all the results above we have that for all $0 < k < N$
$$\e_x\sbrac{Z_k; A}\leq \frac{1+r}{1-r}.$$

Using similar arguments for the process conditioned on $A^c$, 
it follows that for any $0 <x < N$ and any $0 < k \leq x$ we have
\begin{align*}
\e_x\sbrac{Z_k; A^c} & \leq \e_x\sbrac{Z_k \vert A^c} \\
								&=\frac{\bar r+1}{\bar r-1}\frac{(\bar r^{N-k}-1)(\bar r^k-1)}{\bar r^N-1} \\
								& \leq \frac{1+r}{1-r}.
\end{align*}
%
%
Furthermore, for $N > k > x$ we have

\begin{align*}
\e_x\sbrac{Z_k; A^c}&=\e_x\sbrac{Z_k; T_k < T_0 < T_N}\\
										 &=\e_x\sbrac{Z_k \vert T_k < T_0 < T_N} \p_x(T_k < T_0 < T_N)\\
										 &=\e_k\sbrac{Z_k \vert  T_0 < T_N} \p_x(T_k < T_0 < T_N)\\
                                              &\leq\frac{1+r}{1-r}.
\end{align*}

Combining all the above results we have $\e_{\floor{\alpha N}}\sbrac{Z_k} \leq (1+r)/(1-r)$
for all $0 < k < N$. Hence
from \eqref{eq:consensus} we obtain

\begin{align*}
t_N(\alpha)&\leq \frac{2}{q_0+q_1}  \frac{1+r}{1-r} \sum_{k=1}^{N-1} \frac{1}{k},\\
				  & \leq \frac{2}{q_0+q_1}  \frac{1+r}{1-r} (\log(N-1)+1),
\end{align*}
which completes the proof.

\section{Proof of Theorem~\ref{thm:mean_field}}

The process $x^{(N)}(\cdot)$
jumps from the state $x$ to the state $x+1/N$
when 
one of the $N(1-x)$ agents having opinion $\{0\}$
updates (with probability $q_0$) its opinion by interacting with
an agent with opinion $\{1\}$. 
Since the agents update their opinions at points of independent
unit rate Poisson processes,
the rate at which one of the $N(1-x)$ agents having opinion $\{0\}$ 
decides to update its opinion is $N(1-x)q_0$. 
The probability with which the updating agent interacts with an agent
with opinion $\{1\}$ is $x$. Hence, the total rate of transition from $x$
to $x+1/N$ is given by  $r(x \to x+1/N)= {q_0 N x(1-x)}$.
%
%
Similarly, the rate of transition from $x$ to $x-1/N$
is given by $r(x \to x-1/N)= {q_1 N x(1-x)}$.
%
%
From the above transition rates
it can be easily seen  that 
the generator of the process $x^{(N)}(\cdot)$
converges uniformly as $N \to \infty$ to the generator of the
deterministic process $x(\cdot)$ defined by \eqref{eq:mean_field}.
%
From the classical results (see e.g., Kurtz~\cite{Kurtz_1970}), 
the theorem follows.

\section{Proof of Lemma~\ref{lem:monotone}}

We can write $g_K(x)=\phi(\psi(x))$, where $\psi(x)=\frac{x}{1-x}:[0,1) \to [0,\infty)$ and 
$$\phi(t)=\frac{\sum_{i=K+1}^{2K}  \binom{2K}{i} t^i}{\sum_{i=K+1}^{2K}  \binom{2K}{i} t^{2K+1-i}}=
\frac{\sum_{i=K+1}^{2K}  \binom{2K}{i} t^i}{\sum_{i=1}^{K}  \binom{2K}{i-1} t^{i}}.$$ 
Clearly, $\psi(x):[0,1) \to [0,\infty)$  is strictly increasing. Thus, it is sufficient to show that
$\phi:(0,\infty) \to (0,\infty)$ is also strictly increasing.
Clearly, $\phi'(t)=A(t)/(\sum_{i=1}^{K}  \binom{2K}{i-1} t^{i})^2$,
where 
\begin{align}
A(t)&=\brac{\sum_{i=K+1}^{2K}  i\binom{2K}{i} t^{i-1}}\brac{\sum_{i=1}^{K}  \binom{2K}{i-1} t^{i}} \nonumber\\
       &-\brac{\sum_{i=K+1}^{2K}  \binom{2K}{i} t^{i}}\brac{\sum_{i=1}^{K}  i\binom{2K}{i-1} t^{i-1}}\nonumber\\
       &= \sum_{j=K+1}^{3K-1} M_j t^j \label{eq:jbound}
\end{align}
with
$$M_j=\sum_{i=\min(1,j+1-2K)}^{\max(K,j-K)} \binom{2K}{i-1} \binom{2K}{j-i+1}(j+1-2i).$$
We note that in the above sum the running variable $i$ satisfies $i \leq \max(K,j-K)$.
Furthermore, from \eqref{eq:jbound}, we have that $K+1 \leq j \leq 3K-1$.
Hence, we have $i \leq \max(K,j-K) < \frac{j+1}{2}$ for any $K \geq 1$. 
This implies that $M_j > 0$ for all $j$ satisfying $K+1 \leq j \leq 3K-1$
Hence, $\phi'(t) > 0$, $\forall t > 0$,
which implies that $\phi(t)$ is strictly increasing in $(0,\infty)$. 

\section{Proof of Theorem~\ref{thm:phase_transit}}
\label{pf:phase_transit}

From the first step analysis of the embedded chain $\tilde X^{(N)}(\cdot)$ 
it follows that
%
%

\begin{equation}
E_N(n)=p_{n,n+1} E_N(n+1)+p_{n,n-1} E_N(n-1),
\label{eq:first_step}
\end{equation}
which upon rearranging gives

\begin{equation}
E_N(n+1)-E_N(n)= \frac{p_{n,n-1}}{p_{n,n+1}}(E_N(n)-E_N(n-1)).
\label{eq:rec_exit}
\end{equation}
Putting $D_N(n)=E_N(n+1)-E_N(n)$ we find that~\eqref{eq:rec_exit}
reduces to a first order recursion in $D_N(n)$ which satisfies the following
relation for $1 \leq n \leq N-1$

\begin{equation}
D_N(n)=r \frac{1}{g_K(n/N)} D_N(n-1).
\label{eq:drec}
\end{equation}
To compute $D_N(0)$ we use the boundary conditions $E_N(0)=0$
and $E_N(N)=1$, which imply that $\sum_{n=0}^{N-1} D_N(n)=1$.
Hence, we have

\begin{equation}
D_N(0)=\frac{1}{\sum_{t=0}^{N-1} \prod_{j=1}^t \frac{r}{g_K(j/N)}},
\end{equation} 

Thus, using $E_N(n)=\sum_{k=0}^{n-1} D_N(k)$ we have the required
expression for $E_N(n)$ for all $0 \leq n \leq N$.
%

It is also important to note that $D_N$ defines a probability distribution
on the set $\cbrac{0,1,\ldots,N-1}$. Furthermore, using the monotonicity of $g_K$
proved in Lemma~\ref{lem:monotone}
and \eqref{eq:drec} we have

\begin{align*}
D_N(n) & < D_N(n-1) \text{ for } n \geq \floor{\beta N}+1,\\
D_N(n) & > D_N(n-1) \text{ for } n \leq \floor{\beta N}.
\end{align*}
Thus, the mode of the distribution $D_N$ is at $\floor{\beta N}$.
Now for any $\alpha > \beta$ we choose $\beta'$
such that $\alpha > \beta' > \beta$.
Hence, by the monotonicity of $g_K$ we have
$$r':=\frac{r}{g_K(\beta')} < \frac{r}{g_K(\beta)}=1.$$

Also using the monotonicity of $g_K$ and \eqref{eq:drec}
we have for any $j \geq 1$

\begin{align*}
D_N(\floor{\beta' N}+j) & \leq \brac{\frac{r}{g_K\brac{\frac{\floor{\beta'N}+1}{N}}}}^j D_N(\floor{\beta' N}) \\
& \leq (r')^j D_N(\floor{\beta' N}),
\end{align*}
where the last step follows since $\beta'N < \floor{\beta' N}+1$.
Hence, we have

\begin{align*}
E_N(\alpha)&=\sum_{t=0}^{\floor{\alpha N}-1} D_N(t)\\
				 &=1-\sum_{t=\floor{N\alpha}}^{N-1} D_N(t)\\
				 &\geq 1- D_N(\floor{\beta' N}) (r')^{\floor{\alpha N}-\floor{\beta' N}} \sum_{t=0}^{N-1-\floor{\alpha N}} (r')^t\\
				 &\geq 1- (r')^{\floor{\alpha N}-\floor{\beta' N}} \frac{1-(r')^{N-\floor{\alpha N}}}{1-r'}\\
				 & \to 1 \text{ as } N \to \infty.
\end{align*}
The proof for $\alpha < \beta$ follows similarly.

\section{Proof of Theorem~\ref{thm:majority_consensus}}
\label{pf:majority_consensus}

Let $T=T_0\wedge T_N$ denote the random time to reach 
consensus.
Then we have

\begin{align}
T&=\sum_{n=1}^{N-1}\sum_{j=1}^{Z_n} M_{n,j},
\end{align}
where $Z_n$ denotes the number of visits to state $n$ before absorption
and $M_{n,j}$ denotes the time spent in the $j^{\textrm{th}}$ visit to state $n$.
Clearly, the random variables $Z_n$ and $(M_{n,j})_{j\geq 1}$ are independent with each
$M_{n,j}$ being an exponential random variable with rate $q(n\to n+1)+q(n \to n-1)$.
Using Wald's identity we have

\begin{align}
t_{N}(\alpha)&=\e_{\floor{\alpha N}}\sbrac{T} \nonumber\\
					&=\sum_{n=1}^{N-1} \e_{\floor{\alpha N}}\sbrac{Z_n} \e_{\floor{\alpha N}}\sbrac{M_{n,j}} \nonumber\\
				    &=\sum_{n=1}^{N-1}\frac{\e_{\floor{\alpha N}}\sbrac{Z_n}}{(q(n\to n+1)+q(n \to n-1))}.  \label{eq:master_major}
\end{align}
Below we find lower and upper bounds of $t_N(\alpha)$.
Let $A=\cbrac{\omega: T_N(\omega) < T_0(\omega)}$ denote the event that the Markov chain
gets absorbed in state $N$. 
We have

\begin{align}
\expect[\floor{\alpha N}]{Z_n}&=\expect[\floor{\alpha N}]{Z_n \vert A} \prob[\floor{\alpha N}]{A}\nonumber\\
                     &\hspace{2em}+\expect[\floor{\alpha N}]{Z_n \vert A^c} (1-\prob[\floor{\alpha N}]{A}).
                     \label{eq:num_visits_break_major}
\end{align}

\noindent {\em Lower bound of $t_N(\alpha)$}: 
Applying Markov inequality to the RHS of \eqref{eq:uprate} and \eqref{eq:downrate}
we obtain
\begin{align*}
q(n \to n+1)+q(n \to n-1) 
&\leq (N-n)q_0\frac{2K \frac{n}{N}}{K+1}+nq_1\frac{2K \brac{1-\frac{n}{N}}}{K+1}\\
&= (q_0+q_1)\frac{2K}{K+1} \frac{N(N-n)}{N}.
\end{align*}
Furthermore, as in the case of voter model, we have
\begin{align*}
\expect[\floor{\alpha N}]{Z_n}& \geq  \prob[\floor{\alpha N}]{A}\indic{n \geq \floor{\alpha N}}+ (1-\prob[\floor{\alpha N}]{A})\indic{n \leq \floor{\alpha N}}.
\end{align*}
Using \eqref{eq:master_major} and the above inequalities we obtain
\begin{align*}
t_{N}(\alpha)&\geq \frac{\prob[\floor{\alpha N}]{A}}{q_0+q_1}\frac{K+1}{2K}\sum_{n=\floor{\alpha N}}^{N-1}\brac{\frac{1}{n}+\frac{1}{N-n}}\nonumber \\
						&+ \frac{1-\prob[\floor{\alpha N}]{A}}{q_0+q_1}\frac{K+1}{2K}\sum_{n=	1}^{\floor{\alpha N}}\brac{\frac{1}{n}+\frac{1}{N-n}}\\
					&\geq \frac{1}{q_0+q_1} \frac{K+1}{2K} \sum_{n=1}^{N(\alpha \wedge (1-\alpha))} \frac{1}{n}\\
					& > \frac{1}{q_0+q_1} \frac{K+1}{2K} \log(N(\alpha \wedge (1-\alpha))).
\end{align*}

\noindent {\em Upper bound for $t_{N}(\alpha)$}: 
From \eqref{eq:uprate} and \eqref{eq:downrate} we have

\begin{equation}
q(n \to n+1)+q(n \to n-1) \geq \begin{cases}
														c n, \quad \text{for } \frac{n}{N} \leq \frac{1}{2},\\
														c (N-n), \quad \text{for } \frac{n}{N} > \frac{1}{2},
													\end{cases}
\end{equation}
where $c=q_1 \prob{\text{Bin}\brac{2K,\frac{1}{2}}\geq K+1}$.
Using the above inequalities in \eqref{eq:master_major} we have

\begin{equation}
t_N(\alpha) \leq \sum_{n=1}^{\floor{N/2}} \frac{\expect[\floor{\alpha N}]{Z_n}}{c n}+ \sum_{n=\floor{N/2}+1}^{N-1} \frac{\expect[\floor{\alpha N}]{Z_n}}{c(N-n)}
\end{equation}
Hence, to show that $t_N(\alpha)=\Theta(\log N)$
it is sufficient to show that $\e_{\floor{\alpha N}}\sbrac{Z_n}=O(1)$ for 
all $1 \leq n \leq N-1$.

For the rest of the proof we assume $ \alpha > \beta$.
The case $\alpha < \beta$ can be handled similarly.

Let $x=\floor{\alpha N}$. We first find upper bound of $\expect[x]{Z_n;A}$.
Conditioned on $A$, the embedded chain 
$\tilde X^{(N)}$ is a Markov chain with jump
probabilities given by

\begin{equation}
p^{A}_{n,n+1}=1-p^{A}_{n,n-1}=p_{n,n+1}\frac{\prob[n+1]{A}}{\prob[n]{A}}.
\label{eq:cond_jump_major}
\end{equation}
We have 

\begin{align}
\frac{p^{A}_{n,n-1}}{p^{A}_{n,n+1}}&=\frac{p_{n,n-1}}{p_{n,n+1}} \frac{\prob[n-1]{A}}{\prob[n+1]{A}},\\
													   &\overset{(a)}{=}\frac{r}{g_K\brac{\frac{n}{N}}} \frac{\sum_{t=0}^{n-2} \prod_{j=1}^t \frac{r}{g_K(j/N)}}{\sum_{t=0}^{n} \prod_{j=1}^t \frac{r}{g_K(j/N)}},\\
													   & \overset{(b)}{\leq} \min\brac{1,\frac{r}{g_K\brac{\frac{n}{N}}}} \label{eq:ratio},
\end{align}
where equality (a) follows from \eqref{eq:pdef} and Theorem~\ref{thm:phase_transit}.
Inequality (b) follows from the facts (i) $\prob[n-1]{A} \leq \prob[n+1]{A}$ and (ii)
for a monotonically non-increasing non-negative sequence $(y_n)_{n\geq 1}$ the following inequality holds
$$y_n \frac{\sum_{t=0}^{n-2}\prod_{j=1}^{t} y_j}{\sum_{t=0}^{n}\prod_{j=1}^{t} y_j} \leq 1$$
(follows simply by comparing the terms in the numerator with the middle $n-1$ terms
in the denominator).

Given $A$,
let $\zeta_n$ denote the number of times the embedded chain $\tilde X^{(N)}$ jumps
from $n$ to $n-1$ before absorption. Then as in the voter model we have

\begin{equation}
\label{eq:num_visits_major}
Z_n\vert A=\begin{cases}
			1+\zeta_{n}+\zeta_{n+1},  \text{ for } x \leq n \leq N-1,\\
			\zeta_{n}+\zeta_{n+1},  \text{ for } 1 \leq n < x,
			\end{cases}
\end{equation}
%
where $\zeta_n$ follows the recursion

\begin{equation}
\label{eq:branch_major}
\zeta_{n}=\begin{cases}	
			\sum_{l=0}^{\zeta_{n+1}} \xi_{l,n},  \text{ for } x \leq n \leq N-1,\\
			\sum_{l=1}^{\zeta_{n+1}} \xi_{l,n},  \text{ for } 1 \leq n < x
			\end{cases}
\end{equation}
with $\zeta_N=0$ and
$\xi_{l,n}$ denoting the random number of left-jumps from state $n$ between 
$l^{\textrm{th}}$ and $(l+1)^{\textrm{th}}$ left-jumps from state $n+1$. Clearly 
$(\xi_{l,n})_{l\geq 0}$ are i.i.d  with geometric distribution having mean $p^{A}_{n,n-1}/p^{A}_{n,n+1}$.
Hence, applying Wald's identity to solve the above recursion we have

\begin{equation}
\label{eq:recurse_major}
\e_x\sbrac{\zeta_{n}}=\begin{cases}
			\sum_{t=n}^{N-1} \prod_{i=n}^{t} \frac{p_{i,i-1}^A}{p_{i,i+1}^A}, \quad \text{for } x \leq n \leq N-1\\
			\brac{\prod_{i=n}^{x-1} \frac{p_{i,i-1}^A}{p_{i,i+1}^A}}\e_x\sbrac{\zeta_{x}}, \quad \text{for } 1 \leq n < x.
			\end{cases}
\end{equation}
Now using inequality \eqref{eq:ratio}, monotonicity of $g_K$, and the fact that for $n \geq x=\floor{\alpha N} > \floor{\beta N}$,
$1 > r_\alpha:=r/g_K(\alpha) \geq r/g_K(n/N)$ we have for $n \geq x=\floor{\alpha N}$
\begin{equation}
\e_x\sbrac{\zeta_{n}} \leq r_\alpha+r_\alpha^2+\ldots+r_\alpha^{N-n} \leq \frac{r_\alpha}{1-r_\alpha}.
\label{eq:oneside_major}
\end{equation}
Hence, using \eqref{eq:num_visits_major} we have for $n \geq x=\floor{\alpha N}$

\begin{equation}
\expect[x]{Z_n;A} \leq \expect[x]{Z_n \vert A} \leq \frac{1+r_\alpha}{1-r_\alpha}=O(1).
\end{equation}
For $n < x=\floor{\alpha N}$ we have

\begin{equation}
\expect[x]{\zeta_n} \overset{(a)}{\leq} \expect[x]{\zeta_x} \overset{(b)}{\leq} \frac{r_\alpha}{1-r_\alpha},
\end{equation} 
where (a) follows from \eqref{eq:recurse_major} and \eqref{eq:ratio} and (b)
follows from \eqref{eq:oneside_major}.
Hence, from \eqref{eq:num_visits_major} we have  for $n < x=\floor{\alpha N}$

\begin{equation}
\expect[x]{Z_n;A} \leq \expect[x]{Z_n \vert A} \leq \frac{2r_\alpha}{1-r_\alpha}=O(1).
\end{equation}

Similarly, conditioned on $A^c$ we have

\begin{equation}
\label{eq:num_visits_reverse_major}
Z_n\vert A^c=\begin{cases}
			1+\bar \zeta_{n}+\bar \zeta_{n-1}  \text{ for } 1 \leq n \leq x,\\
			\bar \zeta_{n}+\bar \zeta_{n-1}  \text{ for } x <  n < N-1,
			\end{cases}
\end{equation}
where $\bar \zeta_n$ denotes the number of times 
$\tilde X^{(N)}$ jumps to the right from state $n$
given $A^c$. Hence, $\bar \zeta_n$ follows the recursion
given by

\begin{equation}
\label{eq:branch_reverse_major}
\bar \zeta_{n}=\begin{cases}	
			\sum_{l=0}^{\bar \zeta_{n-1}} \bar \xi_{l,n},  \text{ for } 1 \leq n \leq x,\\
			\sum_{l=1}^{\bar \zeta_{n-1}} \bar \xi_{l,n},  \text{ for } x <  n < N-1
			\end{cases}
\end{equation}
where $\bar \zeta_0=0$ and $\bar \xi_{l,n}$
denotes the random number of right-jumps from state $n$ between 
$l^{\textrm{th}}$ and $(l+1)^{\textrm{th}}$ right-jumps from state $n-1$ given $A^c$. 
Clearly $(\bar \xi_{l,n})_{l\geq 0}$ are i.i.d.  with geometric distribution having mean $p^{A^c}_{n,n+1}/p^{A^c}_{n,n-1}$
where 

\begin{equation}
p^{A^c}_{n,n+1}=1-p^{A^c}_{n,n-1}=p_{n,n+1}\frac{\prob[n+1]{A^c}}{\prob[n]{A^c}}.
\label{eq:cond_jump_reverse_major}
\end{equation}
As before, we have 

\begin{align}
\frac{p^{A^c}_{n,n+1}}{p^{A^c}_{n,n-1}}           &{=}\frac{g_K\brac{\frac{n}{N}}}{r} \frac{\sum_{t=n+1}^{N-1} \prod_{j=t+1}^{N-1} \frac{g_K(j/N)}{r}}{\sum_{t=n-1}^{N-1} \prod_{j=t+1}^{N-1} \frac{g_K(j/N)}{r}} {\leq} \min\brac{1,\frac{g_K\brac{\frac{n}{N}}}{r}} \label{eq:ratio_reverse},
\end{align}
Solving \eqref{eq:branch_reverse_major} using Wald's identity we obtain

\begin{equation}
\label{eq:recurse_reverse_major}
\e_x\sbrac{\bar \zeta_{n}}=\begin{cases}
			\sum_{t=1}^{n} \prod_{i=t}^{n} \frac{p_{i,i+1}^{A^c}}{p_{i,i-1}^{A^c}}, \quad \text{for } 1 \leq n \leq x,\\
			\brac{\prod_{i=x+1}^{n} \frac{p_{i,i+1}^{A^c}}{p_{i,i-1}^{A^c}}}\e_x\sbrac{\bar \zeta_{x}}, \quad \text{for } x <  n < N-1.
			\end{cases}
\end{equation}
For $1 \leq n \leq x$ after some simplification of \eqref{eq:recurse_reverse_major} we obtain

\begin{equation}
\e_x\sbrac{\bar \zeta_{n}}=\frac{\prod_{j=1}^{n} \frac{g_K(j/N)}{r}}{\prod_{j=1}^{N-1} \frac{g_K(j/N)}{r}} \prob[n+1]{A^{c}}\brac{1-\prob[n]{A^c}}.
\label{eq:simple_major}
\end{equation}
We observe that for $j \leq \floor{\beta N}$ we have $\frac{g_K(j/N)}{r} \leq 1$ and using the fact that $g_K(x)=1/g_K(1-x)$
we have $\prod_{j=1}^{N-1} \frac{g_K(j/N)}{r}=1/r^{N-1}$. Hence, using \eqref{eq:simple_major} for $n \leq \floor{\beta N}$
we have 

\begin{equation}
\e_x\sbrac{\bar \zeta_{n}} \leq \frac{\prod_{j=1}^{n} \frac{g_K(j/N)}{r}}{\prod_{j=1}^{N-1} \frac{g_K(j/N)}{r}} \leq r^{N-1} \leq 1.
\end{equation}

Furthermore, for $\floor{\beta N} < n \leq x$ we have 

\begin{equation}
\expect[x]{\bar \zeta_n} \leq \frac{1}{\prod_{j=n+1}^{N-1} \frac{g_K(j/N)}{r}} \overset{(a)}{\leq} 1,
\end{equation}
where (a) follows from the fact that $\frac{g_K(j/N)}{r} \geq 1$ for $j > n > \floor{\beta N}$.
Hence, we have shown that $\expect[x]{\bar \zeta_n}=O(1)$ for $1 \leq n \leq x$.
Now, using \eqref{eq:recurse_reverse_major} and inequality \eqref{eq:ratio_reverse} we have for $x < n < N-1$ that
$\expect[x]{\bar \zeta_n} \leq \expect[x]{\bar \zeta_x}=O(1)$.
Hence, from \eqref{eq:num_visits_reverse_major} we see that $\expect[x]{Z_n;A^c}\leq \expect[x]{Z_n\vert A^c} =O(1)$
thereby completing the proof.

\section{Conclusion}
\label{sec:opinion_conclusion}

In this paper, we analysed the voter model
the majority rule model of social interaction
under the presence of biased and stubborn agents.
We observed that for the voter model
the presence of biased agents, reduces the mean consensus
time exponentially in comparison to the voter model
with unbiased agents. For the majority rule model with biased agents, we 
saw that the network reaches the consensus state with all agents
adopting the preferred opinion only if the initial fraction of agents
having the preferred opinion is more than a certain threshold value.
Finally, we have seen that 
for the majority rule model with stubborn agents
the network exhibits metastability, where it
fluctuates between multiple stable configuration,
spending long intervals in each configuration.

Several interesting directions for future work exist.
For example, the behaviour of random $d$-regular networks
under the biased voter and majority rule models has not been analysed yet.
Furthermore, the effect of the presence of more than two opinion on the opinion
dynamics is unknown. It will be also interesting to study the networks dynamics
under the majority rule model for general network topologies when stubborn agents are present.

\begin{acknowledgements}

RR acknowledges support from the University of Waterloo during various visits and also support from the Matrics grant MTR/2017/000141.

\end{acknowledgements}

%
%

\bibliographystyle{spmpsci}      
\bibliography{load_balance,arpan_thesis}   

\end{document}